\documentclass[letterpaper, 10 pt, conference]{ieeeconf}  % Comment this line out
\IEEEoverridecommandlockouts                              % This command is only
\title{\LARGE \bf
Optimality of Affine Policies in Distributionally Robust Linear-Quadratic Control with Temporally Correlated Noise
}

\author{Jakob Nylöf, Daniel Kuhn, John Lygeros, and Giancarlo Ferrari-Trecate% <-this % stops a space
\thanks{This work was supported as a part of NCCR Automation, a National Centre of Competence in Research, funded by the Swiss National Science Foundation (grant number 51NF40\_225155).}% <-this % stops a space
\thanks{J. Nylöf and G. Ferrari-Trecate are with the Institute of Mechanical Engineering and D. Kuhn is with the College of Management of Technology, École Polytechnique Fédérale de Lausanne (EPFL), CH-1015 Lausanne, Switzerland.
        {\tt \{\href{mailto:jakob.nylof@epfl.ch}{jakob.nylof}, \href{mailto:daniel.kuhn@epfl.ch}{daniel.kuhn}, \href{mailto:giancarlo.ferraritrecate@epfl.ch}{giancarlo.ferraritrecate}\}@epfl.ch}}%
\thanks{J. Lygeros is with the Department of Information
Technology and Electrical Engineering, ETH Zürich, CH-8092 Zürich, Switzerland.
        {\tt \href{mailto:jlygeros@ethz.ch}{jlygeros@ethz.ch}}}%
}

\usepackage{graphics}
\usepackage{graphicx}
\usepackage{amsmath}
\usepackage{amssymb}
\usepackage{algorithm}
\usepackage{algpseudocode}
\usepackage{color}
\usepackage{cite}
\usepackage{cancel}
\usepackage{hyperref}

\DeclareMathOperator*{\argmin}{arg\,min}
\DeclareMathOperator*{\argmax}{arg\,max}
\newcommand{\norm}[1]{\left\lVert#1\right\rVert}

\newcommand{\Tr}{\mathrm{Tr}}

\newtheorem{assumption}{Assumption}%[section]
\newtheorem{theorem}{Theorem}%[section]
\newtheorem{proposition}{Proposition}%[section]

\algtext*{EndFor}
\algrenewcommand\algorithmicindent{1em}

\begin{document}

\maketitle
\thispagestyle{empty}
\pagestyle{empty}

\begingroup
\renewcommand\thefootnote{}
\footnotetext{\scriptsize
© 2026 IEEE. Personal use of this material is permitted.
Permission from IEEE must be obtained for all other uses, in any current
or future media, including reprinting/republishing this material for
advertising or promotional purposes, creating new collective works,
for resale or redistribution to servers or lists, or reuse of any
copyrighted component of this work in other works.}
\addtocounter{footnote}{-1}
\endgroup

%%%%%%%%%%%%%%%%%%%%%%%%%%%%%%%%%%%%%%%%%%%%%%%%%%%%%%%%%%%%%%%%%%%%%%%%%%%%%%%%
\begin{abstract}
% Proposition 4 extends.
We study finite-horizon distributionally robust linear-quadratic control with disturbances that can be arbitrarily correlated in time.
% In contrast to existing problem formulations,
The ambiguity set is modeled as a single 2-Wasserstein ball centered at a nominal elliptically contoured distribution of the disturbances.
Despite the infinite-dimensionality of both the policy space and ambiguity set, we prove that the optimal policy is affine and that the worst-case distribution is an affine push-forward of the nominal distribution. These affine maps can be computed efficiently via a best-response algorithm based on the Frank--Wolfe algorithm. Experiments demonstrate improved out-of-sample performance over LQG and its distributionally robust extensions when the true disturbances are correlated, incurring only a small loss of performance under uncorrelated disturbances.

\end{abstract}

%%%%%%%%%%%%%%%%%%%%%%%%%%%%%%%%%%%%%%%%%%%%%%%%%%%%%%%%%%%%%%%%%%%%%%%%%%%%%%%%
\section{Introduction}

Temporally correlated disturbances arise in many real-world control systems,
including power networks \cite{saffari2024spatiotemporal}, navigation systems \cite{langbein2019navigation} and air traffic management \cite{chaloulos2007effect}.
However, precisely identifying temporal correlations from data can be challenging, even in the presence of a single outlier \cite{Chan1992temporaloutliers}.
Since stochastic control relies on knowledge of the distribution of disturbances, misspecification, especially of its temporal correlations, may significantly degrade control performance.

Distributionally robust control addresses stochastic control problems subject to distributional uncertainty. Such problems are formulated as zero-sum games between the decision maker and a fictitious adversary, metaphorically thought of as \emph{Nature}. The decision maker plays first and selects a control policy in order to minimize the expected control cost. Nature plays second and selects a distribution in an ambiguity set to maximally frustrate the decision maker. Thus, the decision maker minimizes the worst-case expected control cost with respect to all distributions in the ambiguity set.

Although both the policy space and the ambiguity set are infinite-dimensional, Distributionally Robust Linear–Quadratic (DRLQ) control problems often admit simple optimal policies. A distributionally robust extension of the classical Linear-Quadratic Gaussian (LQG) control problem is studied in \cite{taşkesen2023distributionallyrobustlinearquadratic_neurips, taşkesen2023distributionallyrobustlinearquadratic}. As in LQG control, the policy space consists of all causal output-feedback policies. However, the process and measurement noise terms have ambiguous distributions that range over separate 2-Wasserstein balls centered at nominal Gaussian distributions. This gives rise to an ambiguity set that includes non-Gaussian distributions but enforces temporal independence of the disturbances. It can be shown that the optimal controller is affine and thus admits a finite parametrization, while the worst-case distribution is a Gaussian distribution parametrized by its finite-dimensional covariance matrix. Extensions of these structural results to other DRLQ settings are reported in \cite{fochesato2025distributionallyrobustlqgkullbackleibler,lanzetti2025optimalitylinearpoliciesdistributionally}. These findings are also related to the optimality of affine estimators in mean squared error estimation and filtering problems \cite{nguyen2021mean, soroosh2018kalman}.

In the existing DRLQ literature, the optimality of affine policies is usually established under the assumption that the disturbances are temporally independent under all distributions in the ambiguity set. Indeed, the proofs build on the classical LQG theory, where the optimality of affine policies follows from the celebrated separation principle, which holds when the disturbances are Gaussian and independent across time. However, in the presence of temporal correlations, the separation principle breaks down. Moreover, it remains unclear whether robustness to temporally independent noise provides sufficient protection against temporal correlations.

Existing DRLQ formulations typically also assume that all distributions in the ambiguity set have zero means, which, together with temporal independence, causes Nature's problem to decouple over time into independently solvable subproblems~\cite{taşkesen2023distributionallyrobustlinearquadratic}. In contrast, non-zero means introduce cross-terms in the objective, which break temporal separability.

In this work, we relax the assumption of independent, zero-mean disturbances and study DRLQ control problems with temporally correlated disturbances (DRLQ-C). We define the ambiguity set as a single 2-Wasserstein ball centered at an elliptically contoured nominal distribution, which generalizes the standard Gaussian setting. Our ambiguity set also contains distributions with non-zero mean.

Our contributions are threefold. First, we show that in DRLQ-C the zero-sum game between the decision maker and Nature admits a Nash equilibrium. The decision maker’s equilibrium strategy is an affine output-feedback policy, while Nature’s equilibrium strategy is an affine push-forward of the nominal distribution. Since these affine maps admit finite parametrizations, the DRLQ-C problem reduces to a finite-dimensional zero-sum game. Our proof differs from those in the DRLQ literature in that it cannot rely on the separation principle. Moreover, working with one single ambiguity set in DRLQ-C allows us to handle the temporal coupling introduced by distributions with non-zero means. Second, we show that the Nash equilibrium can be computed efficiently via a Frank--Wolfe algorithm, which can be interpreted as an iterated best-response algorithm. Third, we demonstrate through numerical experiments that DRLQ-C improves the out-of-sample performance in the presence of correlated noise by (i) achieving lower out-of-sample cost than LQG and DRLQ when the disturbances are correlated under the true distribution, and (ii) remaining only mildly conservative when the disturbances are uncorrelated.

DRLQ-C differs fundamentally from prior DRLQ formulations in that it permits arbitrary temporal dependencies among the disturbances. Note that temporally correlated noise can be incorporated into LQG by representing the disturbances as outputs of an auxiliary linear system driven by independent noise. By augmenting the state, classical DRLQ can thus account for serial correlations in the true disturbances while maintaining the independence of the auxiliary disturbances. However, this approach restricts the admissible noise correlations to those generated by a linear system and fails to capture nonlinear dependencies.

Several works in distributionally robust control impose an affine control parameterization without establishing its optimality 
\cite{kargin2024wasserstein_regret_ih,brouillon2024dro_ih,bartvanparys2016dro_constrainted_stochastic_systems,cescon2025datadrivendistributionallyrobustcontrol}, while others focus on moment ambiguity sets \cite{hakobyan2024wasserstein_DRO}, which provide a coarser description of distributional uncertainty. Establishing the optimality of affine policies under general conditions therefore naturally complements the aforementioned works, particularly when effective computational methods are already available.

\textit{Notation:}
For any $T \in \mathbb{N}$, we define $[T]:=\{0,\dots,T\}$. %Denote by $(x_1,\dots, x_m)$ the column stacked vector of $x_1, \dots, x_m \in \mathbb{R}^n$. 
If $A\in\mathbb{R}^{n\times m}$, then $\|A\|_\mathrm{F}$ denotes its Frobenius norm, $\Tr(A)$ its trace and $\mathrm{vec}(A)$ its column-wise vectorization with inverse $\mathrm{vec}^{-1}$. The Kronecker product of $A$ with $B \in \mathbb{R}^{p \times q}$ is denoted by $A \otimes B$. We use $\operatorname{diag}(A_1,\dots,A_n)$ to denote the block-diagonal matrix with diagonal blocks $A_1,\dots,A_n$. The sets of symmetric, positive semidefinite, and positive definite matrices in $\mathbb{R}^{n\times n}$ are denoted by $\mathbb{S}^n$, $\mathbb{S}_+^n$, and $\mathbb{S}_{++}^n$, respectively, and $I_n$ stands for the identity matrix in~$\mathbb{R}^n$.
% , or $I$ when dimensions are unambiguous.
The positive semidefinite square root of $A\in\mathbb{S}_+^n$ is denoted by $A^{1/2}$. We use $\mathcal P(\mathbb R^n)$ to denote the set of all probability distributions on $(\mathbb R^n,\mathcal B(\mathbb R^n))$ with finite second moments.
% We identify each $\mathbb{P} \in \mathcal P(\mathbb R^n)$ with the law of the identity map on $\mathbb{R}^n$ and thus refer to the elements of $\mathcal{P}(\mathbb R^n)$ as probability distributions.
We write $\mathbb{P} \sim (\mu, \Sigma)$ if $\mathbb{P} \in \mathcal{P}(\mathbb{R}^n)$ has mean vector $\mu$ and covariance matrix $\Sigma$,
%We work with the canonical random vector $\xi(x) = x$ on $\mathbb R^n$, and refer to elements of $\mathcal P_2(\mathbb R^n)$ as distributions.
and we use $\mathcal N(\mu,\Sigma)$ to denote the Gaussian distribution with mean vector $\mu$ and covariance matrix $\Sigma$. For any $\mathbb{P}\in\mathcal{P}(\mathbb{R}^{n})$ and Borel function $f:\mathbb{R}^{n}\to\mathbb{R}^{m}$, $f_\#\mathbb{P}$ stands for the push-forward of $\mathbb{P}$ under $f$.
\section{Problem Formulation}

We consider the linear time-varying system
\begin{equation}
    \label{eq:system_finite_horizon}
\begin{aligned}
    &x_{t+1} = A_tx_t + B_tu_t + D_tw_t, \quad y_t = C_tx_t + v_t,
\end{aligned}
\end{equation}
for $t\in[T-1]$, where $T \in \mathbb{N}$ denotes the time horizon, $x_t\in \mathbb{R}^{n}, \ u_t\in \mathbb{R}^{m}, \ w_t\in \mathbb{R}^{r}, \ y_t\in \mathbb{R}^{p}$ and $v_t \in \mathbb{R}^{p}$ denote the state, control input, process noise, output and measurement noise, respectively,
and $A_t, B_t, C_t, D_t$ are matrices of appropriate dimensions.
The initial state $x_0$ and the noise vectors $w_t$ and $v_t$, $t \in [T-1]$, constitute decision-\emph{in}dependent exogenous random vectors. We combine them to a stacked vector $\xi :=\left(x_0, w_0, v_0,w_1,v_1,\ldots, w_{T-1}, v_{T-1}\right)$ of dimension $N_{\xi}:=n+rT + pT$.
% , is identified with the identity map on $\mathbb{R}^{N_\xi}$ and has
% % identified with the identity map on $\mathbb{R}^{N_\xi}$
% an unknown distribution in $\mathcal{P}(\mathbb{R}^{N_\xi})$.
The decision-dependent endogenous random vectors are the states $x_t$, $t\in[T]$, as well as the inputs $u_t$ and outputs $y_t$, $t\in[T-1]$.
% , defined as Borel measurable functions on $\mathbb R^{N_\xi}$ satisfying \eqref{eq:system_finite_horizon}.
We combine them to stacked vectors 
$x  :=\left(x_0, \ldots, x_T\right) \in \mathbb{R}^{N_x}$, 
$u  :=\left(u_0, \ldots, u_{T-1}\right) \in \mathbb{R}^{N_u}$ and $y :=\left(y_0, \ldots, y_{T-1}\right) \in \mathbb{R}^{N_y}$, where $N_x:=n (T+1)$, $N_u:=mT$, and $N_y:=pT$. We emphasize that $u$, $x$, and $y$ constitute functions of~$\xi$. However, we suppress this dependency notationally to avoid clutter. The system equations \eqref{eq:system_finite_horizon} can now be written compactly as
$x = Hu + D\xi$ and $y = Cx + E\xi$,
where $H \in \mathbb{R}^{N_x \times N_u},$ $C \in \mathbb{R}^{N_y \times N_{x}}$, $D \in \mathbb{R}^{N_x \times N_{\xi}}$ and $E \in \mathbb{R}^{N_y \times N_{\xi}}$ are defined in Appendix \ref{appdx:matrices}.
% We restrict attention to control inputs $u$ causally measurable with respect to the outputs,
Control inputs must be causal,
that is, for every $t \in [T-1]$, $u_t$ must be measurable with respect to the $\sigma$-algebra generated by $y_0,\dots, y_t$.
Equivalently, there must be some Borel function $\pi_t: \mathbb{R}^{p(t+1)} \to \mathbb{R}^{m}$ such that $u_t = \pi_t(y_0,\dots, y_t)$.
% Equivalently, $u$ is induced by a causal output-feedback policy, i.e., for every $t \in [T-1]$, $u_t = \pi_t(y_0,\dots, y_t)$ for some Borel measurable function $\pi_t: \mathbb{R}^{p(t+1)} \to \mathbb{R}^{m}$.
Let $\mathcal N_y$ be the set of all causal inputs~$u$.
% Define the set of causal control policies in the observations $y$ as \begin{equation*}
% \mathcal{N}_y
% := \left\{\, u : \mathbb{R}^{N_y} \to \mathbb{R}^{N_u} \ :\ 
% \begin{aligned}
% &\forall t \in [T],\ \exists\ \text{Borel function } \\
% &\varphi_t : \mathbb{R}^{n_y(t+1)} \to \mathbb{R}^{n_u} \ \text{s.t.}\\
% &u_t(y) = \varphi_t(y_0,\ldots,y_t)
% \end{aligned}
% \,\right\}
% \end{equation*}
% %i.e., $u\in \mathcal{N}_y$ if and only if there exist Borel measurable functions $\varphi_t:\mathbb{R}^{n_y(t+1)}\to \mathbb{R}^{n_u}$ for all times $t$ such that $u_t = \varphi_t(y_0,\dots, y_{t})$, 
% i.e., $u\in \mathcal{N}_y$ if and only if $u_t$ is a causal measurable function of the history $(y_0,\dots,y_{t})$ for all $t\in[T]$.

We assume throughout the paper that the distribution $\mathbb P$ of the exogenous uncertainties~$\xi$ is ambiguous and is only known to fall inside the 2-Wasserstein ambiguity set
\begin{align*}
    \label{eq:Wasserstein_ambiguity_set}
    \mathcal{W} = \left\{\mathbb{P}\in \mathcal{P}(\mathbb{R}^{N_\xi}):\mathbb{W}(\mathbb{P}, \hat{\mathbb{P}})\leq\rho \right\}
\end{align*}
of radius $\rho \geq 0$ around a nominal distribution $\hat{\mathbb{P}} \in \mathcal{P}(\mathbb{R}^{N_\xi})$. The 2-Wasserstein distance between $\mathbb{P} , \hat{\mathbb{P}} \in \mathcal{P}(\mathbb{R}^{N_\xi})$ is
$$
\mathbb{W}(\mathbb{P}, \hat{\mathbb{P}})
:=
\bigg(
\inf_{\pi \in \Pi(\mathbb{P}, \hat{\mathbb{P}})}
\int_{\mathbb{R}^{N_\xi} \times \mathbb{R}^{N_\xi}}
\| \xi - \xi' \|_2^2
\, \mathrm{d}\pi(\xi, \xi')
\bigg)^{1/2},
$$
where $\Pi(\mathbb{P}, \hat{\mathbb{P}})$ denotes the set of distributions on $\mathbb{R}^{N_\xi}\times\mathbb{R}^{N_\xi}$ under which $\xi$ and $\xi'$ have marginals $\mathbb{P}$ and $\hat{\mathbb{P}}$, respectively \cite[Definition 2.18]{kuhn2024distributionallyrobustoptimization}. Unlike the product ambiguity sets over independent, zero-mean noise distributions considered in \cite{taşkesen2023distributionallyrobustlinearquadratic,taşkesen2023distributionallyrobustlinearquadratic_neurips}, $\mathcal{W}$ accommodates distributions of the exogenous noise $\xi$ with non-zero mean and general intertemporal dependencies.
%Moreover, being a single Wasserstein ball without any additional constraints rather than a product of constrained Wasserstein balls, one could argue that \eqref{eq:Wasserstein_ambiguity_set} is the more natural choice for the ambiguity set.

We use $\mathcal W$ to formulate the DRLQ-C model as
\begin{equation}
    \label{eq:DRC_y}
    %\tag{}
    \begin{aligned}
      \inf_{u,x,y} \quad &\sup_{\mathbb{P} \in \mathcal{W}} \ \mathbb{E}_{\mathbb{P}}\left[x^\top Qx + u^\top Ru\right]\\
    \text{s.t.} \quad & u \in \mathcal{N}_y, \ x = Hu + D\xi, \ y = Cx + E\xi,
    \end{aligned}
\end{equation}
where $Q \in \mathbb{S}_{+}^{N_x}$ and $R \in \mathbb{S}_{++}^{N_u}$ represent the state and input cost matrices. Problem $\eqref{eq:DRC_y}$ is distributionally robust in the sense that it minimizes the worst-case expected control cost with respect to all distributions in $\mathcal{W}$. Note that \eqref{eq:DRC_y} can be viewed as a zero-sum game between the decision maker, who plays first, and Nature, who reacts.
By allowing for temporally correlated disturbances in $\mathcal{W}$, DRLQ-C gives Nature more power than the DRLQ models in \cite{taşkesen2023distributionallyrobustlinearquadratic_neurips, taşkesen2023distributionallyrobustlinearquadratic}. We impose the following assumption on $Q$ and $D$.
\begin{assumption}
\label{assm:process_noise_cost}
$Q^\frac{1}{2}D \neq 0$.
\end{assumption}
As $D$ captures the dependence of the state~$x$ on~$\xi$, the condition $Q^\frac{1}{2}D \neq 0$ guarantees that the state cost defined via $Q$ is sensitive to~$\xi$. Using $\hat{\mu} := \mathbb{E}_{\hat{\mathbb{P}}}[\xi]$ and $\hat{\Sigma} := \mathrm{Cov}_{\hat{\mathbb{P}}}[\xi]$ as shorthands for the mean and covariance of matrix of $\xi$ under $\hat{\mathbb{P}}$, we impose the following two assumptions on $\hat{\mathbb{P}}$.
\begin{assumption}
\label{assm:P_hat_positive_covariance}
    $\hat{\Sigma}\in \mathbb{S}_{++}^{N_\xi}$.
\end{assumption}

\begin{assumption}
\label{assm:P_hat_elliptical}
The nominal distribution $\hat{\mathbb{P}}$ is elliptically contoured in the sense of \cite[Definition~3.1]{hult2002multivariate}.
\end{assumption}

Examples of elliptically contoured distributions include Gaussian distributions, Laplace distributions or uniform distributions on ellipsoids, among many others.

Problem~\eqref{eq:DRC_y} is non-convex because the constraint $u \in \mathcal{N}_y$ constitutes a non-convex equality constraint in~$u$ and~$y$. To convexify~\eqref{eq:DRC_y}, we introduce the purified output process $\eta = (\eta_{0},\dots, \eta_{T-1})\in\mathbb{R}^{N_y}$ as in \cite{ben2005control} defined as $\eta := F\xi$ with $F := CD + E \in \mathbb{R}^{N_y \times N_\xi}$; see Appendix~\ref{appdx:matrices} for details. We denote by $\mathcal{N}_\eta$ the set of all control inputs that are causal with respect to~$\eta$. That is, $u\in\mathcal N_\eta$ if and only if for every $t \in [T-1]$, $u_t$ is measurable with respect to the $\sigma$-algebra generated by $\eta_0,\dots, \eta_t$. Equivalently, there must be some Borel function $\tau_t:\mathbb{R}^{p(t+1)} \to \mathbb{R}^{m}$ such that $u_t = \tau_t(\eta_0,\dots, \eta_t)$. One can show that $\mathcal{N}_y = \mathcal{N}_\eta$ \cite[Proposition II.1]{hadjiyiannis2011efficient}. Therefore, $\mathcal{N}_y$ can be replaced by $\mathcal{N}_\eta$ in \eqref{eq:DRC_y} without changing the problem. As $\eta$ is exogenous (independent of $u$), this coordinate transformation renders the variable~$y$ and the constraint $y = Cx + E\xi$ redundant. Hence, \eqref{eq:DRC_y} is equivalent to the primal DRLQ-C problem
\begin{equation}
    \label{eq:DRC}
    \tag{$\mathcal{P}$}
    \begin{aligned}
      \inf_{u,x} \quad &\sup_{\mathbb{P} \in \mathcal{W}} \ \mathbb{E}_{\mathbb{P}}\left[x^\top Qx + u^\top Ru\right]\\
    \text{s.t.} \quad & u \in \mathcal{N}_\eta, \ x = Hu + D\xi.
    \end{aligned}
\end{equation}
Note that the constraint $u \in \mathcal{N}_\eta$ is convex in $u$ because $\eta$ is exogenous, and thus \eqref{eq:DRC} is a convex optimization problem.

\section{Optimality of Affine Policies}

% Changing the order of $\inf$ and $\sup$ in \eqref{eq:DRC} yields the dual
% \begin{equation}
%     \label{eq:DDRC}
%     \tag{$\mathcal{D}$}
%     \begin{aligned}
%       \sup_{\mathbb{P} \in \mathcal{W}} \quad \inf \quad &\mathbb{E}_{\mathbb{P}}\left[x^\top Qx + u^\top Ru\right]\\
%     \text{s.t.} \quad & u \in \mathcal{N}_y, \ x = Hu + D\xi, \ y = Cx + E\xi,
%     \end{aligned}
% \end{equation}
% which models the zero-sum game where the decision maker plays second to maximally penalize Nature.

% In this section, we present the main result of the paper. The proofs of all technical results can be found in Appendix~\ref{appdx:proofs}.
It is expedient to introduce the dual DRLQ-C problem 
\begin{equation}
    \label{eq:DDRC}
    \tag{$\mathcal{D}$}
    \begin{aligned}
      \sup_{\mathbb{P} \in \mathcal{W}} \quad \inf_{u,x} \quad &\mathbb{E}_{\mathbb{P}}\left[x^\top Qx + u^\top Ru\right]\\[-1ex]
    \text{s.t.} \quad & u \in \mathcal{N}_\eta, \ x = Hu + D\xi,
    \end{aligned}
\end{equation}
which is obtained from \eqref{eq:DRC} by interchanging the order of minimization and maximization. Here, Nature plays first, and the decision maker reacts. By weak duality, \eqref{eq:DDRC} lower bounds~\eqref{eq:DRC}. In the remainder of this section, we derive an upper bound on \eqref{eq:DRC} and a lower bound on \eqref{eq:DDRC}. Next, we show that these bounds coincide, which will yield a characterization of the optimal solutions to \eqref{eq:DRC} and \eqref{eq:DDRC}. All proofs are relegated to Appendix~\ref{appdx:proofs}.

\subsection{Upper Bound on \eqref{eq:DRC}}
We proceed in two steps to derive an upper bound on \eqref{eq:DRC}. First, we restrict \eqref{eq:DRC} to minimize only over affine purified output-feedback policies $u = U\eta + q$, where $q\in \mathbb{R}^{N_u}$ and 
\begin{equation}
\label{eq:blocklower}
    U = \begin{bmatrix}
        U_{0,0} \\
        \vdots &  \ddots \\
        U_{T-1,0} & \cdots & U_{T-1,T-1}  \\
    \end{bmatrix}\in \mathbb{R}^{N_u\times N_y}
\end{equation}
is a block lower triangular matrix with $U_{t,s} \in \mathbb{R}^{m \times p}$ for $0 \leq s \leq t \leq T-1$. In the following, we use $\mathcal{U}$ to denote the set of all $(q,U) \in \mathbb{R}^{N_u} \times \mathbb{R}^{N_u \times N_y}$ with $U$ satisfying \eqref{eq:blocklower}. Second, we relax $\mathcal{W}$ to the Gelbrich ambiguity set
\vspace{-0.2em}
\begin{align*}
    \mathcal{G} := \left\{\mathbb{P}\in \mathcal{P}(\mathbb{R}^{N_\xi}): \ \left(\mathbb{E}_\mathbb{P}[\xi], \ \mathbb{E}_\mathbb{P}[\xi\xi^\top]\right) \in \mathcal{M} \right\},
\end{align*}
which is defined in terms of the moment uncertainty set
\begin{align*}
&\mathcal{M}:=\left\{(\mu, M) \in \mathbb{R}^{N_\xi} \times \mathbb{S}_{+}^{N_\xi}: \begin{aligned}
&\exists \Sigma \succeq 0,
\ M=\Sigma+\mu \mu^{\top} \\
&\mathbb{G}^2((\mu, \Sigma),(\hat{\mu}, \hat{\Sigma})) \leq \rho^2\\
\end{aligned}\right\}
\end{align*}
and the (squared) Gelbrich distance
\[
    \mathbb{G}^2((\mu, \Sigma), (\hat\mu, \hat \Sigma)) := \norm{\mu - \hat{\mu}}_2^2 + \Tr[\Sigma + \hat{\Sigma} - 2(\hat{\Sigma}^\frac{1}{2}\Sigma\hat{\Sigma}^\frac{1}{2})^\frac{1}{2}]
\]
on $\mathbb{R}^{N_\xi} \times \mathbb{S}_+^{N_\xi}$ \cite[Definition 2.1]{kuhn2024distributionallyrobustoptimization}. One can show that the Wasserstein ambiguity set $\mathcal{W}$ is contained in the Gelbrich ambiguity set $\mathcal{G}$ \cite[Theorem 2.20]{kuhn2024distributionallyrobustoptimization}. Hence, restricting the primal feasible set to affine policies and relaxing the dual feasible set to $\mathcal{G}$ yields the following upper bound on~\eqref{eq:DRC}:
\begin{equation}
    \label{eq:pre_DRCu}
    \begin{array}{c@{~}l}
      \displaystyle \inf_{u,x,q,U} & \displaystyle \sup_{\mathbb{P} \in \mathcal{G}} \ \mathbb{E}_{\mathbb{P}}\left[x^\top Qx + u^\top Ru\right]\\[2ex]
    \text{s.t.} & (q,U) \in \mathcal{U}, \ x = Hu + D\xi, \ u = UF\xi + q.
    \end{array}
\end{equation}
% where $\mathcal{U} := \{(q,U) \in \mathbb{R}^{N_u} \times \mathbb{R}^{N_u \times N_y}: U \text{ is of form } \eqref{eq:blocklower}\}$.
Next, we introduce the compact notation
\[
    K(U) := \begin{bmatrix}
    R^\frac{1}{2}UF \\
    Q^\frac{1}{2}(HUF + D)
    \end{bmatrix} \quad \text{and} \quad  L := \begin{bmatrix}
    R^\frac{1}{2} \\Q^\frac{1}{2}H
    \end{bmatrix}
\]
to define the function $J: \mathcal{U} \times \mathcal{M} \to \mathbb{R}_+$ via $$\begin{aligned}J(q, U, \mu,M) :=
\ &\mathrm{Tr}\left( K^\top(U) K(U)M\right )\\ &+2q^\top L^\top K(U) \mu + q^\top L^\top L q.\end{aligned}$$
% \begin{equation}
% \label{eq:def_g}
% \begin{aligned}
%     J(q, U, \mu,M) :=
%     &\ \mathrm{Tr}\left( K^\top(U) K(U)M\right )\\
%     &+2q^\top L^\top K(U) \mu + q^\top L^\top L q.
% \end{aligned}   
% \end{equation}
Note that the expected cost in \eqref{eq:pre_DRCu} under $\mathbb{P} \sim (\mu, M - \mu \mu^\top)$ is given by $J(q, U, \mu,M)$. Note also that $\mathbb{P} \in \mathcal{G}$ if and only if $(\mu, M) \in \mathcal{M}$. These observations allow us to reformulate~\eqref{eq:pre_DRCu} as a finite-dimensional minimax problem over $(q,U)$ and $(\mu,M)$. As $\lambda_{\min}(\hat{\Sigma}) > 0$ by Assumption \ref{assm:P_hat_positive_covariance}, the inner maximization over $\mathcal{M}$ can be restricted without loss of optimality to
%the resulting minimax problem can be further upper bounded by restricting~$\mathcal M$ to
\[
    \mathcal{M}^+ := \left\{(\mu, M) \in \mathcal{M}: M - \mu \mu^\top \succeq \lambda_{\min}(\hat{\Sigma})I_{N_\xi} \right\}.
\]
This reasoning culminates in the following proposition.

\begin{proposition}
    \label{prop:ub}
    If Assumptions \ref{assm:process_noise_cost} and \ref{assm:P_hat_positive_covariance} hold, then the primal DRLQ-C problem~\eqref{eq:DRC} is bounded above by
    \begin{equation}
    \label{eq:DRCu}
    \tag{$\mathcal{P}_u$}
    \begin{aligned}
      \min_{(q,U) \in \mathcal{U}} &\max_{(\mu, M) \in \mathcal{M}^+} J(q, U, \mu, M),
    \end{aligned}
\end{equation}
where both the minimum and the maximum are attained.~$\hfill\Diamond$
\end{proposition}

\subsection{Lower Bound on \eqref{eq:DDRC}}
We call $\mathcal{T}:\mathbb{R}^{N_\xi} \to \mathbb{R}^{N_\xi}$ a positive definite (PD) affine map if $\mathcal{T}(\xi) = P\xi + b$ for some $P \in \mathbb{S}_{++}^{N_\xi}$ and $b\in \mathbb{R}^{N_\xi}$.
By defining the restriction of $\mathcal{W}$ given by
\[
    \mathcal{W}_\# := \left\{\mathcal{T}_\#\hat{\mathbb{P}}: \mathcal{T} \text{ is a PD affine map}, \ \mathbb{W}(\mathcal{T}_\#\hat{\mathbb{P}}, \hat{\mathbb{P}})\leq\rho \right\},
\]
the following problem provides a lower bound on~\eqref{eq:DDRC}:
\begin{equation}
    \label{eq:pre_DDRCl}
    \begin{aligned}
      \sup_{\mathbb{P} \in \mathcal{W}_\#} \quad \inf_{u,x} \quad &\mathbb{E}_{\mathbb{P}}\left[x^\top Qx + u^\top Ru\right]\\[-1ex]
    \text{s.t.} \quad & u \in \mathcal{N}_\eta, \ x = Hu + D\xi.
    \end{aligned}
\end{equation}
If $\hat{\mathbb{P}}$ is elliptically contoured, then so is every other $\mathbb{P} \in \mathcal{W}_\#$ \cite[Lemma 3.1]{hult2002multivariate}. In this case, the inner minimization problem in \eqref{eq:pre_DDRCl} constitutes a stochastic LQG problem that is solved by an affine purified output-feedback policy. These insights can be used to prove the following proposition.

\begin{proposition}
    \label{prop:lb}
    If Assumptions \ref{assm:P_hat_positive_covariance} and \ref{assm:P_hat_elliptical} hold, then the dual DRLQ-C problem \eqref{eq:DDRC} is bounded below by
    \begin{equation}
    \label{eq:DDRCl}
    \tag{$\mathcal{D}_l$}
    \begin{aligned}
      \max_{(\mu, M) \in \mathcal{M}^+} \min_{(q,U) \in \mathcal{U}} \ J(q,U, \mu, M),
    \end{aligned}
\end{equation}
where both the maximum and the minimum are attained.~$\hfill\Diamond$
\end{proposition}

We emphasize that, in contrast to the extant DRLQ literature, we cannot invoke the separation principle to establish affine optimality in the inner minimization problem in~\eqref{eq:pre_DDRCl}.

\subsection{Strong Duality and Primal-Dual Solvability}

We are now ready to present the main result. Note that
\begin{equation}
\label{eq:bounds}
\sup \eqref{eq:DDRCl} \leq \sup\eqref{eq:DDRC} \leq \inf \eqref{eq:DRC} \leq \inf \eqref{eq:DRCu},
\end{equation}
where the inequalities follow from Proposition \ref{prop:lb}, weak duality and Proposition \ref{prop:ub}, respectively. Problems \eqref{eq:DRCu} and \eqref{eq:DDRCl} differ only in the order of minimization and maximization. Next, we show that the optimal values of \eqref{eq:DRCu} and \eqref{eq:DDRCl} coincide, which results in the following main theorem.
% The theorem below asserts that strong duality holds, affine policies are optimal in \eqref{eq:DRC} and PSD affine push-forwards of $\hat{\mathbb{P}}$ are optimal in \eqref{eq:DDRC}.
%  The next theorem shows that the primal-dual pairs \eqref{eq:DRCu}, \eqref{eq:DDRCl} and \eqref{eq:DRC}, \eqref{eq:DDRC} both enjoy strong duality. Moreover, \eqref{eq:DRC} is solved by a causal affine policy, while \eqref{eq:DDRC} is solved by a positive semidefinite affine push-forward of $\hat{\mathbb{P}}$.
% % , i.e., an elliptical distribution of the same type.
% Consequently, these solutions form a Nash equilibrium of the zero-sum game associated with \eqref{eq:DRC}.

\begin{theorem}
\label{thm:equality}
    If Assumptions \ref{assm:process_noise_cost}, \ref{assm:P_hat_positive_covariance} and \ref{assm:P_hat_elliptical} hold, then the inequalities in \eqref{eq:bounds} collapse to equalities, and the following~hold.
    \begin{enumerate}
    \renewcommand{\labelenumi}{(\roman{enumi})}
        \item Strong duality holds, that is, $\inf\eqref{eq:DRC} = \sup\eqref{eq:DDRC}$.
        \item If $(q^\star, U^\star)$ solves \eqref{eq:DRCu} and $G^\star = (I + U^\star CH)^{-1}$, then problem \eqref{eq:DRC} is solved by
        \begin{align*}u^\star = U^\star \eta + q^\star = G^\star (U^\star y + q^\star).
        \end{align*}
        \item If $(\mu^\star, M^\star)$ solves \eqref{eq:DDRCl}, $\Sigma^\star = M^\star - \mu^\star \mu^{\star\top}$ and $\mathcal{T}(\xi) = \hat{\Sigma}^{-\frac{1}{2}}(\hat{\Sigma}^{\frac{1}{2}}\Sigma^\star\hat{\Sigma}^{\frac{1}{2}})^\frac{1}{2}\hat{\Sigma}^{-\frac{1}{2}}(\xi - \hat{\mu}) + \mu^\star$, then problem \eqref{eq:DDRC} is solved by $\mathbb{P}^\star = \mathcal{T}_\#\hat{\mathbb{P}} \sim (\mu^\star, \Sigma^\star)$.~$\hfill\Diamond$
    \end{enumerate}
\end{theorem}
\vspace{0.5em}
Theorem \ref{thm:equality} implies that the pair $(u^\star, \mathbb{P}^\star)$ forms a Nash equilibrium of the zero-sum game associated with~\eqref{eq:DRC}. Observe that (ii) expresses the optimal policy both in terms of the purified outputs $\eta$ and the original outputs $y$. Regarding (iii), note that an affine push-forward of an elliptically contoured distribution $\hat{\mathbb{P}}$ remains within the same family of elliptically contoured distributions; see \cite[Lemma~3.1]{hult2002multivariate}. In particular, if $\hat{\mathbb{P}} = \mathcal{N}(\hat{\mu}, \hat{\Sigma})$, then $\mathbb{P}^\star = \mathcal{N}(\mu^\star, \Sigma^\star)$.

%%%%%%%%%%%%%%%%%%%%%%%%%%%%%%%%%%%%%%%%%%%%%%%%%%%%%%%%%%%%%%%%%%%%%%%%%%%%%%%%
\section{Iterated Best-Response Algorithm}
\label{sec:best_response}

We now show that the solution to the upper bounding problem \eqref{eq:DRCu} is the best response to the solution of its dual \eqref{eq:DDRCl} and vice versa. We further show that any solution of~\eqref{eq:DDRCl} allows us to construct solutions of the original primal and dual problems~\eqref{eq:DRC} and~\eqref{eq:DDRC}. We propose to solve problem \eqref{eq:DDRCl} using a Frank–Wolfe algorithm, which admits an equivalent interpretation as an iterated best-response algorithm.

We say that $(q,U)$ is the decision maker's best response to $(\mu, M)$ if $(q,U) \in \argmin_{(q,U) \in \mathcal{U}}J(q,U, \mu, M)$. The next proposition characterizes the decision maker’s best response to $(\mu, M)$ as the solution of a linear system and shows that \eqref{eq:DRCu} is solved by the best response to some $(\mu^\star, M^\star)$.
% \begin{corollary}
% \label{cor:optimal_controller_given_mean_covariance}
%     Let Assumptions \ref{assm:P_hat_positive_covariance} and \ref{assm:P_hat_elliptical} hold and let $(\mu^\star, M^\star)$ solve problem \eqref{eq:DDRCl}, with $\Sigma^\star = M^\star - \mu^\star \mu^{\star\top}$. Then the optimal controller $u^*$ solving problem \eqref{eq:DRC} is given by the causal affine policy $u^* = U^*\eta + q^*$, where
% \begin{equation}
%     \begin{aligned}
% \label{eq:primal_best_response}
%     &q^* = -(U^*F + (R+H^\top Q H)^{-1}H^\top QD)\mu^*\\
%     &U^* = \mathrm{vec^{-1}}\left(\left[J^{-1}Z^\top\left(ZJ^{-1}Z^\top\right)^{-1}ZJ^{-1} - J^{-1})\right]C\right),
% \end{aligned}
% \end{equation}
% where $J := \left(F\Sigma^*F^\top\right) \otimes \left(R + H^\top Q H\right)$, $C := \mathrm{vec}(H^\top QD\Sigma^*F^\top)$ and $Z$ is a full row rank selection matrix that encodes block-lower-triangularity of $U$ with the constraint $Z\mathrm{vec}(U) = 0$. Moreover, $u^*$ can be reparametrized as the causal affine output-feedback controller $u^* = (I + U^*CH)^{-1}U^* y + (I + U^*CH)^{-1}q$.
% \end{corollary}
\begin{proposition}
\label{prop:optimal_controller_given_mean_covariance}
    Assume that $(\mu, M) \in \mathcal{M}^+$, and define $\Sigma = M - \mu \mu^\top$. Then, the decision maker's unique best response $(q^\star,U^\star) \in \mathcal{U}$ to $(\mu, M)$ is given by
\begin{equation}
    \begin{aligned}
\label{eq:primal_best_response}
    &U^\star = \mathrm{vec^{-1}}\left(S \theta^\star\right),\\
    &q^\star = -(U^\star F + (R+H^\top Q H)^{-1}H^\top QD)\mu,
\end{aligned}
\end{equation}
where $S$ is the embedding matrix satisfying $\mathrm{vec}(U) = S\theta$, $\theta$ collects all entries of the blocks $U_{t,s}$ for $0\leq s \leq t \leq T-1$ in \eqref{eq:blocklower}, and $\theta^\star$ is the unique solution to the linear equation
\vspace{-0.5em}

\small
\begin{equation}
\label{eq:theta}
S^{\top}\left[F\Sigma F^\top \otimes \left(R + H^\top Q H\right)\right]S \theta = -S^\top\mathrm{vec}(H^\top Q D \Sigma F^\top).
\end{equation}
\normalsize
Moreover, if Assumptions \ref{assm:process_noise_cost}, \ref{assm:P_hat_positive_covariance} and \ref{assm:P_hat_elliptical} hold and $(\mu^\star, M^\star)$ solves \eqref{eq:DDRCl}, then $(q^\star, U^\star)$ as given in \eqref{eq:primal_best_response} with $(\mu, M) = (\mu^\star, M^\star)$ is the unique solution of \eqref{eq:DRCu}.~$\hfill\Diamond$
% \begin{equation}
%     \begin{aligned}
% \label{eq:primal_best_response}
%     &(R+H^\top Q H)q^\star  = - (H^\top QD + (R+H^\top Q H)U^\star F)\mu^\star,\\
%     &\Pi_{\mathcal{L}}((R + H^\top Q H)U^\star F\Sigma^\star F^\top + H^\top Q D \Sigma F) = 0.
% \end{aligned}
% \end{equation}
% where $U^\star$ solves $J := \left(F\Sigma^*F^\top\right) \otimes \left(R + H^\top Q H\right)$, $C := \mathrm{vec}(H^\top QD\Sigma^*F^\top)$ and $Z$ is a full row rank selection matrix that encodes block-lower-triangularity of $U$ with the constraint $Z\mathrm{vec}(U) = 0$. Moreover, $u^*$ can be reparametrized as the causal affine output-feedback controller $u^* = (I + U^*CH)^{-1}U^* y + (I + U^*CH)^{-1}q$.
\end{proposition}

We say that $(\mu, M)$ is Nature's best response to $(q,U)$ if $(\mu,M) \in \argmax_{(\mu,M) \in \mathcal{M}^+}J(q,U, \mu, M)$.  Before characterizing Nature's best response, we first provide a quasi closed-form solution of an auxiliary optimization problem.
\begin{proposition}[Extension of {\cite[Lemma 3]{nguyen2021meancov}}]
\label{prop:linear_moment_objective_gelbrich_constraint}
    Consider the optimization problem
\begin{equation}
\label{eq:linear_moment_objective_gelbrich_constraint}
    \begin{aligned}\sup_{\mu \in \mathbb{R}^n, \Sigma \succeq 0} \quad  &\mathrm{Tr}\left[\Psi(\Sigma+\mu \mu^\top)\right] + 2\psi^\top \mu \\[-1ex] \mathrm{s.t.} \quad \quad & \mathbb{G}((\mu, \Sigma), (\hat{\mu}, \hat{\Sigma}))\leq \rho
    \end{aligned}
\end{equation}
where $\Psi \in \mathbb{S}_+^n$ and $\psi \in \mathbb{R}^n$.
If $\Psi \neq 0$, $\hat{\Sigma} \in \mathbb{S}_{++}^n$ and $\rho>0$, then problem \eqref{eq:linear_moment_objective_gelbrich_constraint} has a unique maximizer $(\mu^\star, \Sigma^\star)$ given~by
\begin{equation}
\label{eq:solution_dual} 
\begin{aligned}&\mu^\star = (\gamma^\star I_{n} - \Psi)^{-1}( \gamma^\star\hat{\mu} + \psi),\\
&\Sigma^\star=\gamma^{\star 2}\left(\gamma^\star I_n-\Psi\right)^{-1} \hat{\Sigma}\left(\gamma^\star I_n-\Psi\right)^{-1},
\end{aligned}\end{equation}
where $\gamma^\star$ is the unique minimizer of 
\begin{equation}
\label{eq:dual_linear_moment_objective_gelbrich_constraint}
\begin{aligned}
&\inf _{\substack{\gamma>\lambda_{\max }(\Psi)}} (\gamma \hat{\mu} + \psi)^\top (\gamma I_n - \Psi)^{-1}(\gamma \hat{\mu} + \psi) \\
+ \, &\gamma^2\mathrm{Tr}\left[\left(\gamma I_n -\Psi\right)^{-1}\hat{\Sigma}\right]
+ \gamma\left(\rho^2 - \hat{\mu}^\top \hat{\mu} - \mathrm{Tr}(\hat{\Sigma})\right).
\end{aligned}
\end{equation}
Equivalently, $\gamma^\star$ is the unique solution in the interval $\left(\lambda_{\max }(\Psi), \infty\right)$ of the first-order optimality condition
\begin{equation}
\label{eq:optimality_condition_dual_linear_moment_objective_gelbrich_constraint}
\begin{aligned}
\rho^2&-\left\|(\gamma^\star I_n - \Psi)^{-1}(\gamma^\star \hat{\mu} + \psi)
- \hat{\mu}\right\|_2^2 \\& - \mathrm{Tr}\left[\hat{\Sigma}(I_n -\gamma^\star\left(\gamma^\star I_n -\Psi\right)^{-1})^2\right]=0.
\end{aligned}
\end{equation}     
% $\epsilon^2-\|\left(\gamma^* I - D\right)^{-1}(\gamma^* \hat{\mu} + d)
% - \hat{\mu}\|_2^2  - \mathrm{Tr}[\hat{\Sigma}\left(I-\gamma^*\left(\gamma^* I-D\right)^{-1}\right)^2]=0$
 In particular, we have $\Sigma^\star \succeq \lambda_{\min}(\hat{\Sigma})I_n$. Moreover, %$\gamma^*$ satisfies the bounds
 $\underline{\gamma} \leq \gamma^\star \leq \overline{\gamma}$, where
\begin{equation}
\label{eq:bounds_dual_linear_moment_objective_gelbrich_constraint}
\begin{aligned}
&\underline{\gamma} := \lambda_1 + \frac{1}{\rho}\sqrt{|v_1^\top(\lambda_1 \hat{\mu} + \psi)|^2 + \lambda_1^2 v_1^\top \hat{\Sigma} v_1}, \\
&  \overline{\gamma} := \lambda_1 + \frac{1}{\rho}\sqrt{\sum_{i=1}^n |v_i^\top(\lambda_i \hat{\mu} + \psi)|^2 + \lambda_i^2 v_i^\top \hat{\Sigma} v_i},
\end{aligned}
\end{equation}
% $\underline{\gamma} := \lambda_1 + \frac{1}{\epsilon}\sqrt{|v_1^\top(\lambda_1 \hat{\mu} + d)|^2 + \lambda_1^2 v_1^\top \hat{\Sigma} v_1}$ and $\overline{\gamma} \lambda_1 + \frac{1}{\epsilon}\sqrt{\sum_{i=1}^n |v_i^\top(\lambda_i \hat{\mu} + d)|^2 + \lambda_i^2 v_i^\top \hat{\Sigma} v_i}$.
with $\lambda_i$ denoting the $i$-th largest eigenvalue of $\Psi$, and $\{v_i\}_{i=1}^n$ a corresponding orthonormal basis of eigenvectors.~$\hfill\Diamond$
\end{proposition}

% Problem \eqref{eq:linear_moment_objective_gelbrich_constraint} maximizes a function involving the convex quadratic term $\mu^\top \Psi \mu$ and is thus non-convex.
% A convex variant of problem \eqref{eq:linear_moment_objective_gelbrich_constraint} with $\mu = 0$ was studied in \cite[Proposition~A.4]{nguyen2021mean}.
Proposition~\ref{prop:linear_moment_objective_gelbrich_constraint} extends \cite[Lemma~3]{nguyen2021meancov} by providing the bounds on $\gamma^\star$ in \eqref{eq:bounds_dual_linear_moment_objective_gelbrich_constraint}. Problem \eqref{eq:linear_moment_objective_gelbrich_constraint} can therefore be solved efficiently using the following procedure. 
% Proposition~\ref{prop:linear_moment_objective_gelbrich_constraint} shows that problem \eqref{eq:linear_moment_objective_gelbrich_constraint} can be solved efficiently using the following procedure.
First, compute $\gamma^\star$ by solving $\eqref{eq:optimality_condition_dual_linear_moment_objective_gelbrich_constraint}$ via bisection over the interval $[\underline{\gamma}, \overline{\gamma}]$. Next, construct $(\mu^\star, \Sigma^\star)$ using \eqref{eq:solution_dual}. We now use Proposition \ref{prop:linear_moment_objective_gelbrich_constraint} to characterize Nature's best response to $(q, U)$.

\begin{proposition}
\label{prop:mean_covariance_given_optimal_controller}
If Assumption \ref{assm:process_noise_cost} holds, $\rho > 0$, $(q,U) \in \mathcal{U}$, and $(\mu^\star, \Sigma^\star)$ is the unique maximizer of \eqref{eq:linear_moment_objective_gelbrich_constraint} with $\Psi = K^\top(U)K(U)$ and $\psi = K^\top(U) L q$, then Nature's unique best response to $(q,U)$ is the pair $(\mu^\star, M^\star) \in \mathcal{M}^+$, where $M^\star = \Sigma^\star + \mu^\star\mu^{\star\top}$. Moreover, if Assumptions~\ref{assm:process_noise_cost}, \ref{assm:P_hat_positive_covariance} and~\ref{assm:P_hat_elliptical} hold and $(q^\star, U^\star)$ solves \eqref{eq:DRCu}, then $(\mu^\star, M^\star)$ defined as above with $(q, U) = (q^\star, U^\star)$ is the unique solution of \eqref{eq:DDRCl}.~$\hfill\Diamond$
\end{proposition}

By using \cite[Proposition 2.2]{kuhn2024distributionallyrobustoptimization}, the dual problem \eqref{eq:DDRCl} can be reformulated as a monolithic tractable semidefinite program (SDP). Tractability is inherited by \eqref{eq:DRCu}, \eqref{eq:DRC} and~\eqref{eq:DDRC} thanks to Proposition~\ref{prop:optimal_controller_given_mean_covariance} and Theorem~\ref{thm:equality}. However, the dimensions of the decision variables and constraints in this SDP scale as $\mathcal{O}(n T)$. Therefore, even though it can be solved in polynomial time, runtimes grow quickly with $n$ and $T$.

% approach is not computationally efficient for large system dimensions $n$ or long time horizons $T$.

A more efficient method for solving \eqref{eq:DDRCl}  is the classical Frank--Wolfe algorithm \cite{frankwolfe1956}. To describe this method, we define $V(\mu, M)$ as the optimal value function of the inner minimization problem in~\eqref{eq:DDRCl}. Using this notation, \eqref{eq:DDRCl} can be represented more compactly as $\max_{(\mu, M) \in \mathcal{M}^+} V(\mu, M)$. Note that $V(\mu, M)$ is concave and smooth because $J(q,U,\mu,M)$ is affine in~$(\mu,M)$ and strictly convex quadratic in $(q,U)$, respectively. 
%$V(\mu, M) := \min_{(q,U) \in \mathcal{U}} J(q, U, \mu, M) = J(q^*(\mu,M),U^*(\mu,M), \mu, M)$ where $(q^*(\mu,M),U^*(\mu,M))$ is the best response to $(\mu, M)$ as in \eqref{eq:primal_best_response}. 
At each iteration $t \in \mathbb{N}$, the Frank--Wolfe algorithm solves the following direction-finding subproblem, which maximizes the first-order Taylor expansion of $V$ around the current iterate $(\mu_t, M_t) \in~\hspace{-0.3em}\mathcal{M}^+$:
% To this end, we define $V : \mathcal{M}^+ \to \mathbb{R}_+$, $V(\mu, M) := \min_{(q,U) \in \mathcal{U}} J(q, U, \mu, M) = J(q^*(\mu,M),U^*(\mu,M), \mu, M)$ where $(q^*(\mu,M),U^*(\mu,M))$ is the best response to $(\mu, M)$ as in \eqref{eq:primal_best_response}. Note that \eqref{eq:DDRCl} can be written as $\max_{(\mu, \Sigma) \in \mathcal{M}^+} V(\mu, M)$.
% At iteration $t \in \mathbb{N}_{+}$, given $(\mu_t, M_t) \in \mathcal{M}^+$, one maximizes the first-order Taylor expansion of the objective function $V$ around the current iterate $(\mu_t, M_t) \in \mathcal{M}^+$ solves the linearized subproblem
\begin{equation}
    \label{eq:fw_lin_subproblem}
    % \max_{(\mu, \Sigma) \in \mathcal{M}^+} \langle\nabla_{(\mu,M)}f(\mu, M),(\mu, M)\rangle =
    \max_{(\mu, M) \in \mathcal{M}^+}  \nabla_\mu V(\mu_t,M_t)^\top \mu + 
    \mathrm{Tr}(\nabla_M V(\mu_t, M_t) M).
\end{equation}
The next iterate is then constructed as
\[
    (\mu_{t+1}, M_{t+1}) = (\mu_t, M_t) + \alpha_t((\mu_{\mathrm{FW}}, M_{\mathrm{FW}}) - (\mu_t, M_t)),
\]
where $(\mu_{\mathrm{FW}}, M_{\mathrm{FW}})$ denotes the solution of \eqref{eq:fw_lin_subproblem}, and $\alpha_t>0$ stands for the step size. The vanilla Frank--Wolfe algorithm sets $\alpha_t = \frac{2}{t+2}$, while the fully adaptive Frank--Wolfe algorithm determines $\alpha_t$ via line search; see \cite[Algorithm 1]{nguyen2021mean}.
% Since $V$ is concave and smooth while $\mathcal{M}^+$ is convex and compact, one can show that $V$ is Lipschitz smooth on $\mathcal{M}^+$, and thus the vanilla Frank--Wolfe algorithm converges to the solution $(\mu^\star, M^\star)$ of \eqref{eq:DDRCl} at a sublinear rate \cite[Theorem 1]{jaggi2013fw}. If $\hat{\mu} = 0$, one can further show that the fully adaptive Frank--Wolfe algorithm converges to $(\mu^\star, M^\star)$ at a linear rate by adapting the arguments of \cite[Theorem 6.2]{nguyen2021mean}.
Since $V$ is concave and smooth while $\mathcal{M}^+$ is convex and compact, one can show that $V$ is Lipschitz smooth on $\mathcal{M}^+$, and thus the vanilla Frank--Wolfe algorithm achieves sublinear convergence of $V(\mu_t,M_t)$ to the optimal value of~\eqref{eq:DDRCl} \cite[Theorem~1]{jaggi2013fw}. %Since the maximizer is unique, $(\mu_t,M_t)\to(\mu^\star,M^\star)$.
If $\hat{\mu}=0$, the fully adaptive Frank--Wolfe algorithm can be shown to achieve linear convergence in objective value by adapting~\cite[Theorem~6.2]{nguyen2021mean}.
Accordingly, we set $\hat{\mu} = 0$ in all numerical experiments of Section \ref{sec:num}. Empirically, however, we observed that the fully adaptive algorithm converges even when $\hat{\mu}\neq 0$.

Computing the gradients $\nabla_{\mu} V(\mu,M)$ and $\nabla_{M} V(\mu,M)$ in the objective function of~\eqref{eq:fw_lin_subproblem} directly
is difficult because this would require differentiating through the matrix-valued maps $q^\star(\mu, M)$ and $U^\star(\mu, M)$ in \eqref{eq:primal_best_response}.
In the context of simpler estimation problems, a direct (though cumbersome) computation of $\nabla V$ is possible \cite{nguyen2021mean}, whereas \cite{taşkesen2023distributionallyrobustlinearquadratic} relies on automatic differentiation. In contrast, we apply Danskin's theorem to compute $\nabla V$ efficiently in the DRLQ-C setting.

\begin{proposition}
\label{prop:gradient_computation}
    Fix $(\mu, M) \in \mathcal{M}^+$, and let $(q^\star, U^\star)$ be the unique solution to $\min_{(q,U) \in \mathcal{U}} J(q, U, \mu, M)$. Then $V$ is differentiable at $(\mu,M)$, and its gradients are given by
    \vspace{-0.3em}
    \begin{align*}
    &\nabla_{\mu} V(\mu, M) = \nabla_\mu J(q^\star, U^\star, \mu, M) = 2K^\top(U^\star) L q^\star,\\
    &\nabla_{M} V(\mu, M) = \nabla_M J(q^\star, U^\star, \mu, M) = K^\top(U^\star)K(U^\star).~\hfill\Diamond
    \end{align*}
\end{proposition}
\vspace{0.5em}
Using $(q^*(\mu,M),U^*(\mu,M))$ to denote the best response to $(\mu, M)$ defined in \eqref{eq:primal_best_response} and recalling the definition of the value function~$V(\mu,M)$, one readily verifies that the direction-finding subproblem~\eqref{eq:fw_lin_subproblem} has the same unique maximizer as
\begin{equation}
    \label{eq:fw_best_response_subproblem}
    \max_{(\mu, M) \in \mathcal{M}^+}  J(q^\star(\mu_t, M_t), U^\star(\mu_t, M_t), \mu, M).
\end{equation}
Problem \eqref{eq:fw_best_response_subproblem} ostensibly computes Nature's best response to $(q^*(\mu_t,M_t),U^*(\mu_t,M_t))$, which is itself the decision maker’s best response to $(\mu_t,M_t)$. Consequently, the Frank–Wolfe algorithm can be interpreted as an iterated best-response algorithm between the decision maker and Nature, with damped updates for Nature; see Algorithm~\ref{alg:play_fw_game}. The best responses are computed efficiently as described in Propositions~\ref{prop:optimal_controller_given_mean_covariance} and~\ref{prop:mean_covariance_given_optimal_controller}. A similar best-response algorithm is studied in~\cite{shoebi2026}. We emphasize that DRLQ-C particularly benefits from gradient computation via Proposition~\ref{prop:gradient_computation} over the automatic differentiation approach in \cite{taşkesen2023distributionallyrobustlinearquadratic} (see Section \ref{sec:num}), presumably due to the size of $\nabla V$ scaling as $\mathcal{O}(n^2T^2)$.
%\vspace{-1em}
\begin{algorithm}
\caption{Iterated Best-Response Algorithm}
\label{alg:play_fw_game}
\begin{algorithmic}%[1]
\Require %$(\mu_0, M_0) \in \mathcal{M}^+$, 
$\mu_0 \gets \hat{\mu}, \ M_0 \gets \hat{\Sigma} + \hat{\mu}\hat{\mu}^\top$
\For{$t \in \mathbb{N}$ until convergence}
    \State $(q_{t}, U_{t}) \gets \arg\min_{(q,U) \in \mathcal{U}} J(q,U,\mu_t, M_t)$ %\hfill primal best response
    \State $(\mu_\textrm{FW}, M_{\textrm{FW}}) \gets \arg\max_{(\mu, M) \in \mathcal{M}^{+}} J(q_{t}, U_{t}, \mu, M)$ %\hfill dual best response
    \State select a step size $\alpha_t\in[0,1]$  %\hfill set step size
    \State $(\mu_{t+1}, M_{t+1}) \gets (\mu_t, M_t) + \alpha_t ((\mu_\textrm{FW},M_\textrm{FW}) - (\mu_t, M_t))$
\EndFor
\Ensure $(q_t, U_t)$ solves \eqref{eq:DRCu} and $(\mu_t, M_t)$ solves \eqref{eq:DDRCl}
\end{algorithmic}
\end{algorithm}
\vspace{-1em}

\section{Numerical Experiments}
\label{sec:num}

We perform numerical experiments to analyze the DRLQ-C model and its numerical scalability using Algorithm~\ref{alg:play_fw_game}. All experiments were run on a machine with an Apple~M3 chip and 36 GB of RAM. The code is available at \href{https://github.com/DecodEPFL/drlq-c}{\texttt{github.com/DecodEPFL/drlq-c}}.

\subsection{Numerical Scalability}

The first experiment compares the runtime of Algorithm~\ref{alg:play_fw_game} against the time needed by MOSEK to solve the monolithic SDP reformulation of~\eqref{eq:DDRCl} discussed in Section~\ref{sec:best_response}. We consider a family of time-invariant systems with $n = r = d$ and $m = p = \max\{1,\lfloor{d/5}\rfloor\}$, where $d \in \mathbb{N}$ is a tunable parameter. We set $T = 5$, $Q = I_{N_x}$, $R = I_{N_u}$, $D_t = I_r$, and $A_t = \bar{A}/(2\rho(\bar{A}))$, where $\bar{A}$ is a random matrix with independent standard normal entries, and~$\rho(\bar{A})$ denotes its spectral radius. We also set $B_t = U_BV_B^\top$, where $U_B\Lambda_B V_B^\top$ is the singular value decomposition of another random matrix $\bar{B}$ with independent standard normal entries. The matrix~$C_t$ is constructed analogously. We set $\rho = N_\xi^{1/2}$, $\hat{\mu} = 0$ and $\hat{\Sigma} = \mathrm{diag}(\hat{\Sigma}_{x_0}, I_T \otimes \hat{\Sigma}_{wv})$, where $\hat{\Sigma}_{x_0}$ and $\hat{\Sigma}_{wv}$ are generated as follows. For $\alpha \in \{x_0, wv\}$, let $M_\alpha \in \mathbb{R}^{d_\alpha\times d_\alpha}$ with $d_{x_0}=n$ and $d_{wv}=r+p$ be a random matrix with independent standard normal entries, let $P_\alpha$ be the orthogonal matrix of eigenvectors of $(M_\alpha + M_\alpha^\top)/2$, and define $\hat{\Sigma}_{\alpha} = P_\alpha \Lambda_\alpha P_\alpha^\top$, where $\Lambda_\alpha$ is a diagonal matrix with independent diagonal entries drawn uniformly from $[1,2]$. The large size of the resulting instance of~\eqref{eq:DRC} is evident from the matrices in Appendix~\ref{appdx:matrices}, whose dimensions scale as $\mathcal O(dT)$. Figure~\ref{fig:runtime} reports the runtimes of MOSEK as well as the vanilla and the fully adaptive Frank–Wolfe algorithms as a function of~$d$, averaged over 3 independent runs. The gradients in the objective of~\eqref{eq:fw_lin_subproblem} are either computed using Proposition~\ref{prop:gradient_computation} or via automatic differentiation. All algorithms are terminated once the relative duality gap drops below~$10^{-4}$. We observe that the fully adaptive Frank-Wolfe algorithm that computes gradients using Proposition~\ref{prop:gradient_computation} is the fastest. All Frank–Wolfe methods outperform MOSEK, which runs out of memory for~$d>15$. All fully adaptive Frank-Wolfe algorithms are about twenty times faster than their vanilla counterparts, and computing gradients using Proposition~\ref{prop:gradient_computation} is generally twice as fast as automatic differentiation.

\begin{figure}
    \centering
    \includegraphics[width=\linewidth]{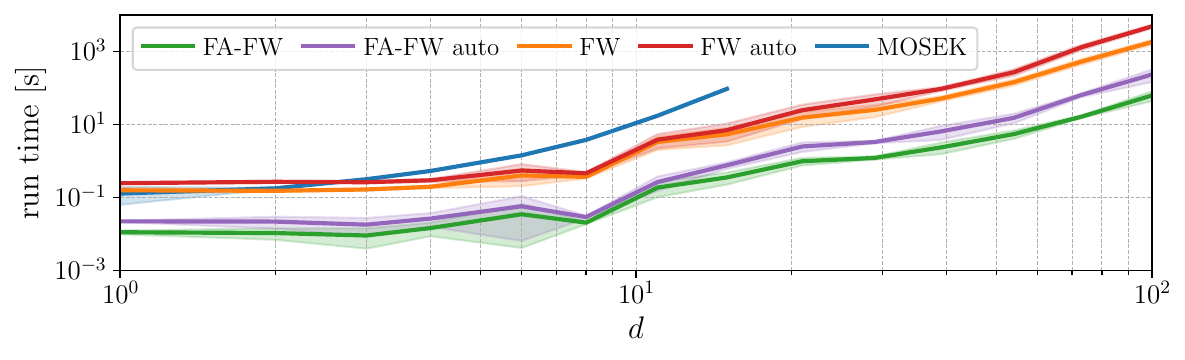}
    \vspace{-2.5em}
    \caption{Runtime as a function of the tuning parameter~$d$ for the fully adaptive Frank--Wolfe algorithm with gradient computation using Proposition~\ref{prop:gradient_computation} (FA--FW) and automatic differentiation (FA--FW auto), the vanilla Frank-Wolfe algorithm with gradient computation using Proposition~\ref{prop:gradient_computation} (FW) and automatic differentiation (FW auto), and the SDP solved by MOSEK.}
    \label{fig:runtime}
    \vspace{-0.5em}
\end{figure}

\begin{figure}
    \centering
    \includegraphics[width=\linewidth]{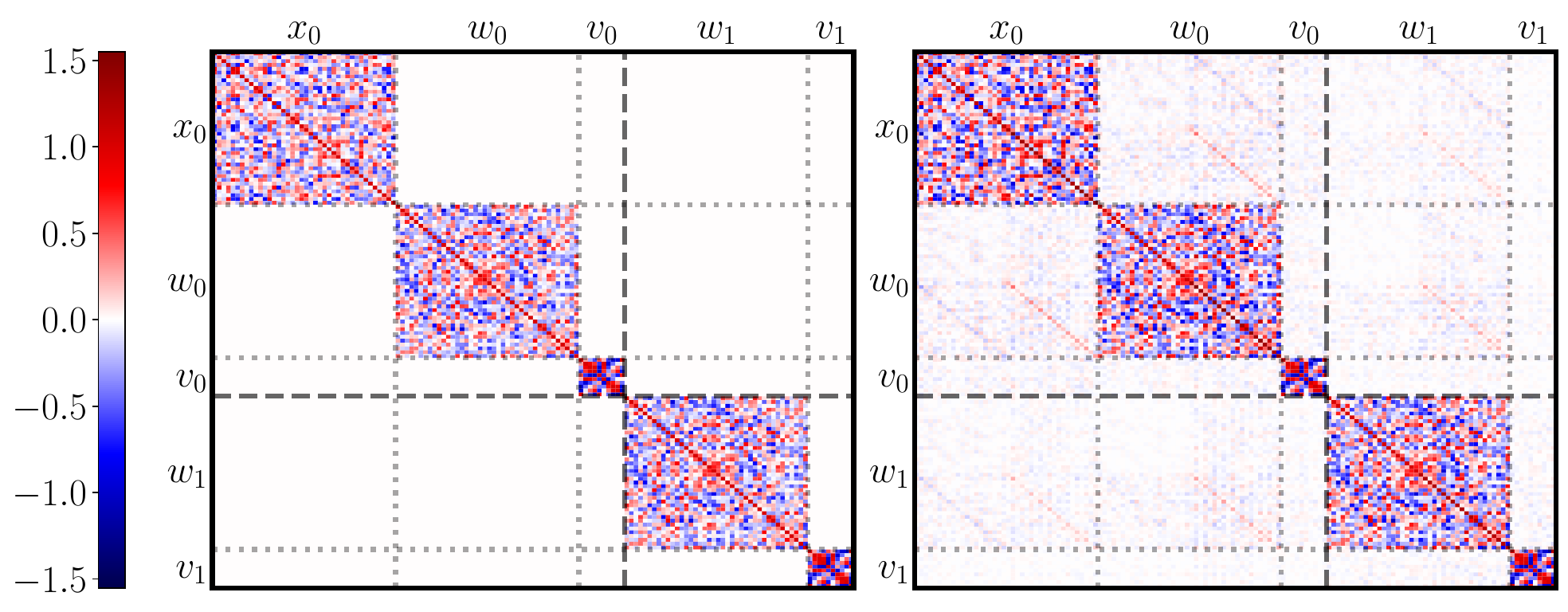}
    \vspace{-1.5em}
    \caption{Nominal covariance matrix $\hat{\Sigma}$ (left) %DRLQ worst-case (middle)
    and worst-case DRLQ-C covariance matrix (right) of $\xi = (
        x_0, w_0, v_0 ,w_1,v_1)$. Dashed and dotted lines separate the correlation blocks associated with different time steps and associated with the initial state, process and measurement noise, respectively.}
    \vspace{-1.5em}
    \label{fig:sigmas}
\end{figure}

\begin{figure}
    \vspace{-1.5em}
    \centering
    \includegraphics[width=\linewidth]{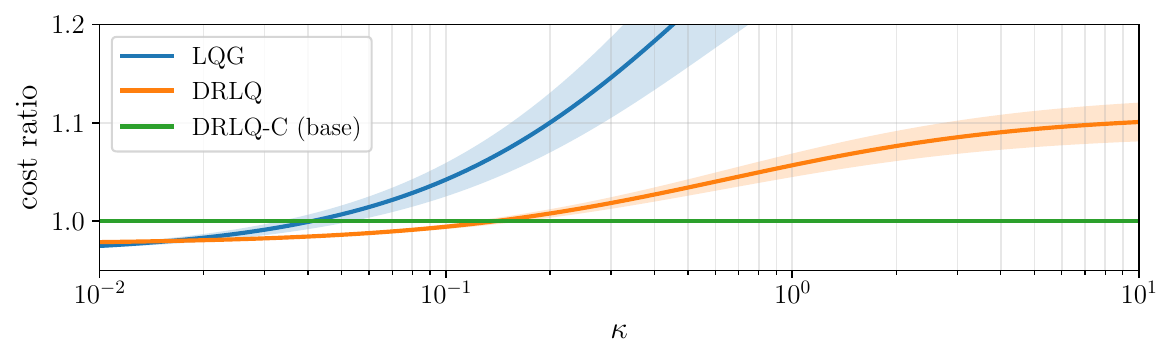}
    \vspace{-2.2em}
    \caption{Normalized costs of the LQG, DRLQ, and DRLQ-C optimal controllers under $\Sigma_{\mathrm{OOS}}(\kappa)$ relative to the DRLQ-C cost for each $\kappa$.}
    \vspace{-1em}
    \label{fig:cost_ratio}
\end{figure}

\subsection{Out-of-Sample Robustness to Correlated Noise}

We compare the out-of-sample performance of the DRLQ-C controller with that of (i) the classical LQG controller and (ii) the DRLQ controller of \cite{taşkesen2023distributionallyrobustlinearquadratic}. As a benchmark, we study the swing dynamics of a networked power system with~$N$ buses and state \(x_t=(\theta_t,\omega_t) \in \mathbb{R}^{2N}\), where $\theta_t$ and $\omega_t$ denote the voltage angles and frequency deviations \cite{bergen2007structure}. The dynamics are discretized with time step $\Delta > 0$ and given by

\vspace{-0.5em}
\small 
\[
A_t\!=\!
\begin{bmatrix}
I_N & \Delta I_N\\
-\Delta M^{-1}L & I_N-\Delta M^{-1}\Gamma
\end{bmatrix},\,
B_t\!=\!
\begin{bmatrix}
0\\
\Delta M^{-1}
\end{bmatrix}, \,
D_t\!=\!
I_{2N},
\]
\normalsize
where $M = \mathrm{diag}(m_1,\dots, m_N) \in \mathbb{S}_{++}^N$ is the inertia matrix, $L$ the graph Laplacian and $\Gamma = \mathrm{diag}(\gamma_1,\dots,\gamma_N) \in \mathbb{S}_{++}^N$ the damping matrix. We use a ring network, that is, a cycle graph with $N=20$ nodes, and we set $T = 2$, $\Delta = 0.5$, $Q = I_{T+1} \otimes \mathrm{diag}(0.1L, 5I_N)$, $R = I_T \otimes 0.1I_N$, $\rho = 1.77$ and $\hat{\mu} = 0$. The system is fully actuated in frequency, while $C_t$ is chosen such that the output consists of noisy frequency measurements at every second bus. To generate a single benchmark instance, we draw $m_i$ and $\gamma_i$ uniformly from $[3,5]$ and $[0.8,1.2]$, respectively, and we set $\hat{\Sigma} = \mathrm{diag}(\hat{\Sigma}_{x_0}, \hat{\Sigma}_{w}, \hat{\Sigma}_{v}, \hat{\Sigma}_{w}, \hat{\Sigma}_{v})$, where $\hat{\Sigma}_{x_0}$, $\hat{\Sigma}_{w}$ and $\hat{\Sigma}_{v}$ are generated as follows. For each $\beta\in\{x_0,w,v\}$, we draw $U_\beta \in \mathbb{R}^{n_\beta \times r_\beta}$ with independent standard normal entries, where $n_{x_0} = n_w = 40$, $n_v = 10$, $r_{x_0} = r_{w} = 8$  and $r_v = 2$. We then set $S_\beta=U_\beta U_\beta^\top+10^{-2}I_{n_\beta}$ and define $\hat\Sigma_\beta$ through $(\hat{\Sigma}_\beta)_{ij} = (S_\beta)_{ij} /\sqrt{(S_\beta)_{ii}(S_\beta)_{jj}}$. Figure~\ref{fig:sigmas} compares~$\hat{\Sigma}$ against the covariance matrix solving~\eqref{eq:DDRCl}, that is, Nature's equilibrium strategy. We observe that Nature introduces relatively small but noticeable temporal correlations relative to the block-diagonal entries. The corresponding optimal mean is $\mu^\star= \hat{\mu} = 0$.

In DRLQ, Nature cannot introduce temporal correlations and thus selects the covariance matrices of $x_0$, $w_0$, $v_0$, $w_1$ and $v_1$ independently; see \cite{taşkesen2023distributionallyrobustlinearquadratic}. We set $\rho_{x_0} = \rho_{w} = 1$ as the radii for $x_0, w_0, w_1$ and $\rho_v = 0.25$ for $v_0, v_1$. We then compute the cost for the optimal LQG, DRLQ and DRLQ-C controllers under zero-mean disturbances with out-of-sample covariance $\Sigma_{\mathrm{OOS}}(\kappa) = \hat{\Sigma} + \kappa \Sigma_{\mathrm{corr}}$, where $\Sigma_\mathrm{corr}$ is defined through $(\Sigma_{\mathrm{corr}})_{ij}
= (S)_{ij} /\sqrt{(S)_{ii}(S)_{jj}}$ with $S=U U^\top+10^{-2}I_{N_\xi}$ and $U \in \mathbb{R}^{N_\xi \times 14}$ drawn with independent standard normal entries, and $\kappa \geq 0$ represents the correlation intensity. Figure~\ref{fig:cost_ratio} shows the out-of-sample costs with respect to $\Sigma_{\mathrm{OOS}}(\kappa)$ for $\kappa \in [10^{-2},10^1]$, averaged over 10 independent draws of $\Sigma_{\mathrm{corr}}$. We observe that DRLQ-C outperforms both other methods for $\kappa > 0.15$, while only being mildly conservative for $\kappa \leq 0.15$.

%%%%%%%%%%%%%%%%%%%%%%%%%%%%%%%%%%%%%%%%%%%%%%%%%%%%%%%%%%%%%%%%%%%%%%%%%%%%%%%%
\section{Conclusion}

We established the optimality of affine policies for distributionally robust linear-quadratic control problems affected by ambiguous and temporally correlated noise. We also showed that the worst-case noise distribution is an affine push-forward of the nominal distribution at the center of the ambiguity set. Both the optimal controller and the worst-case distribution can be computed efficiently, and the optimal controller achieves low out-of-sample cost both under correlated and uncorrelated noise. Future work will investigate ambiguity sets imposing sparse correlation patterns.

%%%%%%%%%%%%%%%%%%%%%%%%%%%%%%%%%%%%%%%%%%%%%%%%%%%%%%%%%%%%%%%%%%%%%%%%%%%%%%%%
% \section{Acknowledgements}

% The authors gratefully acknowledge the contribution of National Research Organization and reviewers' comments.

% %%%%%%%%%%%%%%%%%%%%%%%%%%%%%%%%%%%%%%%%%%%%%%%%%%%%%%%%%%%%%%%%%%%%%%%%%%%%%%%%

% References are important to the reader; therefore, each citation must be complete and correct. If at all possible, references should be commonly available publications.

% \begin{thebibliography}{99}

% \bibitem{c1}
% J.G.F. Francis, The QR Transformation I, {\it Comput. J.}, vol. 4, 1961, pp 265-271.

% \bibitem{c2}
% H. Kwakernaak and R. Sivan, {\it Modern Signals and Systems}, Prentice Hall, Englewood Cliffs, NJ; 1991.

% \bibitem{c3}
% D. Boley and R. Maier, "A Parallel QR Algorithm for the Non-Symmetric Eigenvalue Algorithm", {\it in Third SIAM Conference on Applied Linear Algebra}, Madison, WI, 1988, pp. A20.

% \end{thebibliography}

\bibliographystyle{ieeetr} % Numerisk stil
\bibliography{references} % Din .bib-fil

\appendices
\section{Block Matrices}
\label{appdx:matrices}
The block matrices $H \in \mathbb{R}^{N_x \times N_u}$, $C \in \mathbb{R}^{N_y \times N_{x}}$, $D \in \mathbb{R}^{N_x \times N_{\xi}}$, $E \in \mathbb{R}^{N_y \times N_{\xi}}$ and $F \in \mathbb{R}^{N_y \times N_\xi}$ are defined as
% \footnotesize
% $$
% \begin{aligned}
% & H:=\begin{bmatrix}
% 0 & \\
% \mathcal{A}_1^1 B_0 & 0 & & \\
% \mathcal{A}_1^2 B_0 & \mathcal{A}_2^2 B_1 & 0 & & \\
% \vdots & &  & \ddots & \\
% \vdots & & & & 0\\
% \mathcal{A}_1^T B_0 & \mathcal{A}_2^T B_1 & \cdots & \cdots & \mathcal{A}_T^T B_{T-1}
% \end{bmatrix},\\ &D:=\begin{bmatrix}
% \mathcal{A}_0^0 & \\
% \mathcal{A}_0^1 & \mathcal{A}_1^1 D_0 \\
% \mathcal{A}_0^2 & \mathcal{A}_1^2 D_0 & \mathcal{A}_2^2 D_0 & & \\
% \vdots & & & \ddots & \\
% \mathcal{A}_0^T & \mathcal{A}_1^T D_0 & \mathcal{A}_2^T D_0 & \cdots & \mathcal{A}_T^T D_0
% \end{bmatrix} \\
% & C:=\begin{bmatrix}
% C_0 & 0 & & \\
% & \ddots & \ddots & \\
% & & C_{T-1} & 0
% \end{bmatrix}, \quad E:=\begin{bmatrix}
% 0 & E_0 & & \\
% & 0 & \ddots & \\
% &   & \ddots & \ddots \\
% & & & 0& E_0
% \end{bmatrix}
% \end{aligned}
% $$
% \normalsize
\vspace{-1em}

\footnotesize
\setlength{\arraycolsep}{3pt}
\[
\begin{aligned}
H &:=
\begin{bmatrix}
0 & \cdots & 0 \\
\mathcal{A}_1^1 B_0 &  \\
\vdots  & \ddots & \\
\mathcal{A}_1^T B_0  & \cdots & \mathcal{A}_T^T B_{T-1}
\end{bmatrix},\ C :=
\begin{bmatrix}
C_0 & & & 0\\
& \ddots &  & \vdots \\
& & C_{T-1} & 0
\end{bmatrix},
\end{aligned}
\]
\[
\begin{aligned}
D &:=
\begin{bmatrix}
\mathcal{A}_0^0 \\
\mathcal{A}_0^1 & \mathcal{A}_1^1 \mathcal{D}_0 \\
\vdots & \vdots & \ddots &  \\
\mathcal{A}_0^T & \mathcal{A}_1^T \mathcal{D}_{0}  & \cdots & \mathcal{A}_T^T \mathcal{D}_{T-1}
\end{bmatrix}, \ E :=
\begin{bmatrix}
0 & E_0 \\
\vdots &  & \ddots \\
0 & &  & E_0
\end{bmatrix},\\
F &:=
\begin{bmatrix}
C_0\mathcal{A}_0^0 & E_0 \\
C_1\mathcal{A}_0^1 & C_1\mathcal{A}_1^1 \mathcal{D}_{0} & E_0 \\
\vdots & \vdots & \ddots &\\
C_{T-1}\mathcal{A}_0^{T-1} & C_{T-1}\mathcal{A}_1^{T-1}\mathcal{D}_0
& \cdots 
& C_{T-1}\mathcal{A}_{T-1}^{T-1} \mathcal{D}_{T-2}
& E_0
\end{bmatrix},
\end{aligned}
\]
\normalsize
where $\mathcal{A}_s^t:=A_{t-1} A_{t-2} \cdots A_s$ for $s<t$, $\mathcal{A}_t^t:=I_{n}$, $\mathcal{D}_t = \begin{bmatrix}D_{t} &0\end{bmatrix}$ and $E_0 = \begin{bmatrix}0 & I_{p}\end{bmatrix}$.
% Moreover, the block matrix associated with the purified observation process $\eta = F\xi$ is given by
% \footnotesize
% \begin{equation*}
% F =
% \begin{bmatrix}
% C_0\mathcal{A}_0^0 & E_0 \\
% C_1\mathcal{A}_0^1 & C_1\mathcal{A}_1^1 D_0 & E_0 \\
% \vdots & \vdots & \ddots &\\
% C_{T-1}\mathcal{A}_0^{T-1} & C_1\mathcal{A}_1^{T-1}D_0
% & \cdots 
% & C_{T-1}\mathcal{A}_{T-1}^{T-1} D_0
% & E_0
% \end{bmatrix}
% \end{equation*}
% \normalsize
Note that $F$ has full row-rank.

\section{Proofs}
\label{appdx:proofs}

% \subsection{Proof of Proposition \ref{prop:ub}}
% \label{appdx:ub}

\textit{Proof of Proposition \ref{prop:ub}:}
Problem \eqref{eq:pre_DRCu} has the same optimal value as $\inf_{(q,U) \in \mathcal{U}} \sup_{(\mu, M) \in \mathcal{M}} J(q, U, \mu, M)$.
% \begin{equation}
%     \label{eq:proof_11}
%     \begin{aligned}
%       \inf_{(q,U) \in \mathcal{U}} &\sup_{(\mu, M) \in \mathcal{M}} J(q, U, \mu, M).
%     \end{aligned}
% \end{equation}
We show that the inner supremum $\sup_{(\mu, M) \in \mathcal{M}} J(q, U, \mu, M)$ can be restricted to $\mathcal{M}^+$. The statement is true if $\rho = 0$, since then $\mathcal{M} = \{(\hat{\mu}, \hat{M})\}$ with $\hat{M} - \hat{\mu}\hat{\mu}^\top = \hat{\Sigma} \succeq \lambda_{\min}(\hat{\Sigma})I_{N_\xi}$. Assume that $\rho > 0$. The inner supremum is equivalent to \eqref{eq:linear_moment_objective_gelbrich_constraint} in Proposition \ref{prop:linear_moment_objective_gelbrich_constraint} with $\Psi = K^\top(U)K(U)$, $\psi = K^\top(U)Lq$ and the change of variables $(\mu, \Sigma(\mu, M)) = (\mu, M - \mu \mu^\top)$, where $(\mu, M) \in \mathcal{M}$. The assumptions of Proposition \ref{prop:linear_moment_objective_gelbrich_constraint} are satisfied since $\hat{\Sigma} \succ 0$ by Assumption \ref{assm:P_hat_positive_covariance} and $\Psi \neq 0$ for all $(q,U) \in \mathcal{U}$. Indeed, $\Psi = 0$ if and only if $K(U) = 0$, which holds if and only if $R^{\frac{1}{2}}UF = 0$ and $Q^{\frac{1}{2}}(HUF + D) = 0$. Since $R \succ 0$, the latter conditions imply $Q^\frac{1}{2}D = 0$, contradicting Assumption \ref{assm:process_noise_cost}. Hence, Proposition \ref{prop:linear_moment_objective_gelbrich_constraint} applies. In particular, $\sup_{(\mu, M) \in \mathcal{M}} J(q, U, \mu, M)$ is attained, and its feasible set can be restricted to $\mathcal M^+$.
% Hence, \eqref{eq:proof_11} is equivalent to
% \begin{equation}
%     \label{eq:proof_12}
%     \begin{aligned}
%       \inf_{(q,U) \in \mathcal{U}} &\max_{(\mu, M) \in \mathcal{M}^+} J(q, U, \mu, M).
%     \end{aligned}
% \end{equation}
Since $J(\cdot,\cdot,\mu, M)$ is proper and continuous for every $(\mu, M) \in \mathcal{M}$, the point-wise maximum $W(q, U) := \max_{(\mu, M) \in \mathcal{M}^+}J(q,U,\mu,M)$ is proper and lower semicontinuous. To conclude that $\inf_{(q,U) \in \mathcal{U}} W(q,U)$ is attained, it remains to show that $W$ is coercive.
% \begin{align*}&W(q,U) \geq J(q,U,\hat{\mu},\hat{M})\geq \mathbb{E}_{\hat{\mathbb{P}}}[(UF\xi + q)^\top R (UF\xi + q)] \\
% &= \mathrm{Tr}(F^\top U^\top R UF \Sigma) + (UF\hat\mu + q)^\top R (UF\hat\mu + q) \\
% &\geq \lambda_{\min}(R)\left[\lambda_{\min}(\hat{\Sigma}) \lambda_{\min}(FF^\top) \|U\|_\mathrm{F}^2 + \|UF\hat\mu + q\|_2^2\right]\end{align*}
Indeed, $W(q,U) \geq J(q,U,\hat{\mu},\hat{M})\geq \mathbb{E}_{\hat{\mathbb{P}}}[(UF\xi + q)^\top R (UF\xi + q)]
= \mathrm{Tr}(F^\top U^\top R UF \hat\Sigma) + (UF\hat\mu + q)^\top R (UF\hat\mu + q)
\geq \lambda_{\min}(R)\lambda_{\min}(\hat{\Sigma}) \lambda_{\min}(FF^\top) \|U\|_\mathrm{F}^2 + \lambda_{\min}(R)\|UF\hat\mu + q\|_2^2$ where the first inequality follows since $(\hat{\mu}, \hat{M}) \in \mathcal{M}^+$, the second inequality by $J(q,U,\hat{\mu},\hat{M}) = \mathbb{E}_{\hat{\mathbb{P}}}[x^\top Q x + u^\top R u]$ under affine policies and discarding the first integrand, the first equality by expanding and using that $\hat M = \hat\Sigma + \hat\mu \hat\mu^\top$, and the third inequality by repeated application of $\mathrm{Tr}(A^\top B A) \geq \lambda_{\min}(B)\mathrm{Tr}(A^\top A)$ for $B \succeq 0$ and cyclicity of the trace. It follows that $\lambda_{\min}(R), \lambda_{\min}(\hat{\Sigma}), \lambda_{\min}(FF^\top) > 0$, since $R \succ 0$, $\hat\Sigma \succ 0$ by Assumption \ref{assm:P_hat_positive_covariance} and $F$ has full row-rank. Hence, $W(q,U) \to \infty$ if $\|U\|_\mathrm{F} \to \infty$. It remains to show that $W(q,U) \to \infty$ if $U$ is bounded and $\|q\|_2 \to \infty$. This follows, since then, $\|UF\mu + q\|_2 \to \infty$ as $\|q\|_2 \to \infty$.
% Hence, $W$ is coercive.
\hfill $\QED$

\vspace{0.5em}

\textit{Proof of Proposition \ref{prop:lb}:}
We first define $\mathcal{W}_\#^+ := \{\mathbb{P} \in \mathcal{W}_\#:\mathrm{Cov}_\mathbb{P}[\xi] \succeq \lambda_{\min}(\hat{\Sigma})I_{N_\xi}\}$, obtaining a lower bound on \eqref{eq:pre_DDRCl} by restricting the maximization to $\mathcal{W}_\#^+$. Next, we observe that the inner infimum in this lower bound can be restricted to affine policies. Indeed, every $\mathbb{P} \in \mathcal{W}_\#^+$ has the following two properties: (i) $\mathbb{P}$ is an affine push-forward of $\hat{\mathbb{P}}$ and hence elliptically contoured by Assumption \ref{assm:P_hat_elliptical} \cite[Lemma 3.1]{hult2002multivariate} and (ii) the covariance matrix of $\mathbb{P}$ is positive definite. Hence, \cite[Lemma IV.4, Lemma IV.5]{hadjiyiannis2011efficient} apply and the inner infimum in the lower bound is solved by an affine policy, that is,
\begin{equation}
\label{eq:proof_21}
\sup_{\mathbb{P}\in\mathcal{W}_\#^+}\ \min_{(q,U)\in\mathcal U}
\left\{\mathbb{E}_{\mathbb P}[x^\top Qx+u^\top Ru]:
\begin{aligned}&u=UF\xi+q,\\&x=Hu+D\xi\end{aligned}\right\},
\end{equation}
is a lower bound on \eqref{eq:pre_DDRCl}.
% \begin{equation}
%     \label{eq:proof_21}
%     \begin{aligned}
%       \sup_{\mathbb{P} \in \mathcal{W}_\#^+} \quad \inf \quad &\mathbb{E}_{\mathbb{P}}\left[x^\top Qx + u^\top Ru\right]\\
%     \mathrm{s.t.} \quad & (q,U) \in \mathcal{U}, \ u = UF\xi + q, \\ &x = Hu + D\xi.\\
%     \end{aligned}
% \end{equation}
Let $\mathcal{G}^+ : = \{\mathbb{P} \in \mathcal{P}(\mathbb{R}^{N_\xi}): (\mathbb{E}_\mathbb{P}[\xi], \mathbb{E}_\mathbb{P}[\xi\xi^\top]) \in \mathcal{M}^+\}$. Clearly, $\mathcal{W}_\#^+ \subset \mathcal{G}^+$ by the Gelbrich bound \cite[Theorem 2.20]{kuhn2024distributionallyrobustoptimization}.
% It follows that $\mathcal{W}_\#^+ \subseteq \mathcal{G}^+$. Indeed, let $\mathbb{P} \in \mathcal{W}_\#^+$ and denote its mean and covariance by $(\mu, \Sigma)$. Since $\mathbb{P}$ is a positive semidefinite affine push-forward of $\hat{\mathbb{P}}$ and $\hat{\Sigma} \succ 0$, we have by \cite[Theorem 2]{nguyen2021mean} that $\mathbb{G}((\mu,\Sigma), (\hat\mu, \hat\Sigma)) = \mathbb{W}(\mathbb{P}, \hat{\mathbb{P}})$, and hence $\mathbb{P} \in \mathcal{G}^+$.
It remains to show that relaxing the feasible set $\mathcal{W}_\#^+$ in \eqref{eq:proof_21} to $\mathcal{G}^+$ does not change the optimal value. Indeed, let $\mathbb{P} \in \mathcal{G}^+$ with $\mathbb{P} \sim (\mu, \Sigma)$ and define $\mathcal{T}(\xi) = P(\xi - \hat{\mu}) + b$ with $P = \hat{\Sigma}^{-\frac{1}{2}}(\hat{\Sigma}^{\frac{1}{2}}\Sigma\hat{\Sigma}^{\frac{1}{2}})^\frac{1}{2}\hat{\Sigma}^{-\frac{1}{2}}\in \mathbb{S}_{++}^{N_\xi}$ and $b = \mu$.
% , which is possible by Assumption \ref{assm:P_hat_positive_covariance}.
Then the push-forward $\mathbb{Q} := \mathcal{T}_\#\hat{\mathbb{P}}$ satisfies $\mathbb{Q} \sim (\mu, \Sigma)$. Indeed, $\mathbb{E}_{\mathbb{Q}}[\xi] = \mathbb{E}_{\hat{\mathbb{P}}}[\mathcal{T}(\xi)] = \mu$ and
$\mathbb{E}_{\mathbb{Q}}[(\xi - \mu)(\xi - \mu)^\top] = P\hat{\Sigma} P^\top
%=\hat{\Sigma}^{-\frac{1}{2}}(\hat{\Sigma}^{\frac{1}{2}}\Sigma\hat{\Sigma}^{\frac{1}{2}})^\frac{1}{2}\hat{\Sigma}^{-\frac{1}{2}}\hat{\Sigma}\hat{\Sigma}^{-\frac{1}{2}}(\hat{\Sigma}^{\frac{1}{2}}\Sigma\hat{\Sigma}^{\frac{1}{2}})^\frac{1}{2}\hat{\Sigma}^{-\frac{1}{2}}
%=\hat{\Sigma}^{-\frac{1}{2}}(\hat{\Sigma}^{\frac{1}{2}}\Sigma\hat{\Sigma}^{\frac{1}{2}})^\frac{1}{2}(\hat{\Sigma}^{\frac{1}{2}}\Sigma\hat{\Sigma}^{\frac{1}{2}})^\frac{1}{2}\hat{\Sigma}^{-\frac{1}{2}}
%=\hat{\Sigma}^{-\frac{1}{2}}(\hat{\Sigma}^{\frac{1}{2}}\Sigma\hat{\Sigma}^{\frac{1}{2}})\hat{\Sigma}^{-\frac{1}{2}} 
=\Sigma.$
% \begin{align*}&\mathrm{Cov}_\mathbb{Q}[\xi] = \mathbb{E}_{\mathbb{Q}}[(\xi - \mu)(\xi - \mu)^\top] = P\hat{\Sigma} P^\top\\
% &=\hat{\Sigma}^{-\frac{1}{2}}(\hat{\Sigma}^{\frac{1}{2}}\Sigma\hat{\Sigma}^{\frac{1}{2}})^\frac{1}{2}\hat{\Sigma}^{-\frac{1}{2}}\hat{\Sigma}\hat{\Sigma}^{-\frac{1}{2}}(\hat{\Sigma}^{\frac{1}{2}}\Sigma\hat{\Sigma}^{\frac{1}{2}})^\frac{1}{2}\hat{\Sigma}^{-\frac{1}{2}}\\
% &=\hat{\Sigma}^{-\frac{1}{2}}(\hat{\Sigma}^{\frac{1}{2}}\Sigma\hat{\Sigma}^{\frac{1}{2}})^\frac{1}{2}(\hat{\Sigma}^{\frac{1}{2}}\Sigma\hat{\Sigma}^{\frac{1}{2}})^\frac{1}{2}\hat{\Sigma}^{-\frac{1}{2}}\\
% &=\hat{\Sigma}^{-\frac{1}{2}}(\hat{\Sigma}^{\frac{1}{2}}\Sigma\hat{\Sigma}^{\frac{1}{2}})\hat{\Sigma}^{-\frac{1}{2}} =\Sigma.
% \end{align*}
% It follows that $\mathrm{Cov}_\mathbb{Q}[\xi] = \Sigma \succeq \lambda_{\min}(\hat{\Sigma})I$, where the inequality is due to $\mathbb{P} \in \mathcal{G}^+$.
Thus, we have $\mathbb{W}(\mathbb{Q}, \hat{\mathbb{P}}) = \mathbb{G}((\mu, \Sigma), (\hat{\mu}, \hat{\Sigma})) \leq \rho$, where the equality is due to \cite[Theorem 2]{nguyen2021mean} and the inequality follows since $\mathbb{P} \in \mathcal{G}^+$. Hence, $\mathbb{Q} \in \mathcal{W}_\#^+$. The cost function in \eqref{eq:proof_21} depends only on first and second moments of $\mathbb{P}$. Thus, for every $\mathbb{P} \in \mathcal{G}^+$, there exists $\mathbb{Q}\in \mathcal{W}_{\#}^+$ %with the same first and second moments, and therefore 
with the same objective value as $\mathbb{P}$. Problem \eqref{eq:DDRCl} is now obtained from \eqref{eq:proof_21} by replacing $\mathcal{W}_\#^+$ with $\mathcal{G}^+$ and reformulating in terms of $(\mu, M) \in \mathcal{M}^+$. The supremum in \eqref{eq:DDRCl} is attained, as it maximizes an upper semicontinuous function on the compact set $\mathcal{M}^+$. Indeed, $\mathcal{M}$ is compact \cite[Proposition 2.3]{kuhn2024distributionallyrobustoptimization} and
$\mathcal M^+$ is a closed subset of $\mathcal{M}$. \hfill \QED

\vspace{0.5em}

\textit{Proof of Theorem \ref{thm:equality}:} To prove (i), it suffices to show the statement for \eqref{eq:DRCu} and \eqref{eq:DDRCl}. Note that $J(q,U,\mu,M)$ is concave (affine) in $(\mu,M)$. We show that $J$ is convex in $(q,U)$. Indeed, with $M = \Sigma + \mu\mu^\top$, we have $J(q,U,\mu,M) = \mathrm{Tr}(K^\top(U)K(U)\Sigma) + \|K(U)\mu + Lq\|_2^2$, which is convex since $K(U)$ is affine and $\Sigma \succeq 0$. Moreover, $\mathcal U$ is convex and $\mathcal M$ is convex and compact \cite[Proposition 2.3]{kuhn2024distributionallyrobustoptimization}, and therefore so is $\mathcal M^+$.
% , being a closed subset of $\mathcal{M}$ and an intersection of convex sets.
% Hence, $\mathcal M^+$ is convex, being the intersection of two convex sets. Moreover, $\mathcal M^+$ is a closed subset of the compact set $\mathcal M$ and is therefore compact.
Now apply \cite[Theorem 4.2]{sion1958general}.

To prove (ii), denote by $\mathcal{J}(u, \mathbb{P})$ the cost function in \eqref{eq:DRC}. It holds that $u^\star = U^\star\eta + q^\star$ is optimal in \eqref{eq:DRC}. Indeed, $\sup_{\mathbb{P} \in \mathcal{W}}\mathcal J(u^\star, \mathbb{P}) \leq \sup_{\mathbb{P} \in \mathcal{G}}\mathcal J (u^\star, \mathbb{P}) = \sup_{(\mu, M) \in \mathcal{M}^+} J(q^\star,U^\star, \mu, M) = \inf_{u \in \mathcal{N}_\eta} \sup_{\mathbb{P} \in \mathcal{W}} \mathcal{J}(u, \mathbb{P})$. The inequality is due to $\mathcal{W} \subset \mathcal{G}$. The first equality follows by reformulating in terms of $(\mu, M) \in \mathcal{M}$ and noting that $\mathcal{M}$ can be restricted to $\mathcal{M}^+$ by the same argument as in the proof of Proposition \ref{prop:ub}. The second equality is the equality of \eqref{eq:DRC} and \eqref{eq:DRCu} by (i). The affine representation in $y$ is due to \cite{SkafBoyd2010}.

To prove (iii), note that $\mathbb{P}^\star \sim (\mu^\star, \Sigma^\star)$ by the definition of $\mathcal{T}$ and the same computation as in the proof of Proposition \ref{prop:lb}. We show that $\mathbb{P}^\star$ is optimal in \eqref{eq:DDRC}. Indeed, $\inf_{u \in \mathcal{N}_\eta} \mathcal{J}(u, \mathbb{P}^\star) = \inf_{(q, U) \in \mathcal{U}} \mathcal{J}(U F \xi + q, \mathbb{P}^\star) = \inf_{(q, U) \in \mathcal{U}} J(q, U, \mu^\star, M^\star) = \sup_{\mathbb{P} \in \mathcal{W}} \inf_{u \in \mathcal{N}_\eta} \mathcal{J}(u, \mathbb{P})$.
The first equality follows by affine optimality \cite[Lemma IV.4, Lemma IV.5]{hadjiyiannis2011efficient} and the second equality from reformulation in terms of moments. The third equality is the equality of \eqref{eq:DDRC} and \eqref{eq:DDRCl} by (i). \hfill \QED

\vspace{0.5em}

\textit{Proof of Proposition \ref{prop:optimal_controller_given_mean_covariance}:}
Minimizing $J(q, U, \mu, M)$ first over $q$ yields the optimizer $q^\star(U,\mu, M)$ given in the proposition statement. Substituting $q^\star$ back and vectorizing yields the problem $\min_{(\cdot,U) \in \mathcal{U}} \mathrm{vec}(U)^\top [(F\Sigma F^\top) \otimes (R + H^\top Q H)]\mathrm{vec}(U) + 2\mathrm{vec}(U)^\top \mathrm{vec}(H^\top Q D \Sigma F^\top)$, where $\Sigma = M - \mu\mu^\top$. Note that since $\Sigma \succeq \lambda_{\min}(\hat{\Sigma})I_{N_\xi}$, $F$ has full row-rank and $R \succ 0$, the Kronecker product in the objective is positive definite. Hence, with the injective parametrization $\mathrm{vec}(U) =S\theta$,
% , where $\theta$ is the vector of block lower triangular entries of $U$,
the unique optimizer $\theta^\star$ is the unique solution to the first-order optimality condition \eqref{eq:theta}, and $U^\star = \mathrm{vec}^{-1}(S\theta^*)$. To prove the second assertion, note that Theorem \ref{thm:equality} (i) applies by assumption. By solvability of \eqref{eq:DRCu} and uniqueness of best responses, it follows that the unique optimal solution of \eqref{eq:DRCu} is the best response to $(\mu^*, M^*)$.
\hfill \QED

\vspace{0.5em}

\textit{Proof of Proposition \ref{prop:linear_moment_objective_gelbrich_constraint}:} The form of the maximizer $(\mu^\star, \Sigma^\star)$ of \eqref{eq:linear_moment_objective_gelbrich_constraint} given in \eqref{eq:solution_dual} is proven in \cite[Lemma 3]{nguyen2021meancov}. The bound $\Sigma^\star \succeq \lambda_{\min}(\hat{\Sigma})I_n$ is proven in \cite[Proposition A.4]{nguyen2021mean}. It remains to prove the bounds on $\gamma^\star$ given in \eqref{eq:bounds_dual_linear_moment_objective_gelbrich_constraint}. Note that $(\gamma^\star I_n-\Psi)^{-1}
(\gamma^\star\hat\mu+\psi)-\hat\mu=(\gamma^\star I_n-\Psi)^{-1}
(\Psi\hat\mu+\psi)$. Moreover, since $\Psi = \sum_{i=1}^n\lambda_iv_i v_i^\top$, it follows that $(\gamma^\star I_n - \Psi)^{-1} = \sum_{i=1}^n(\gamma^\star - \lambda_i)^{-1}v_i v_i^\top$. Thus, $\|
(\gamma^\star I_n-\Psi)^{-1}
(\gamma^\star\hat\mu+\psi)-\hat\mu\|_2^2
=
\sum_{i=1}^n
a_i
(\gamma^\star-\lambda_i)^{-2}
$, with $a_i := |v_i^\top(\lambda_i\hat\mu+\psi)|^2$. Similarly, $\mathrm{Tr}[\hat\Sigma(I_n - \gamma^\star(\gamma^\star I_n - \Psi)^{-1})^2] = \mathrm{Tr}[\hat\Sigma(-\Psi(\gamma^\star I_n - \Psi)^{-1})^2] = \sum_{i=1}^n b_i(\gamma^\star - \lambda_i)^{-2}$, where $b_i:=\lambda_i^2 v_i^\top \hat{\Sigma}v_i$. Equation \eqref{eq:optimality_condition_dual_linear_moment_objective_gelbrich_constraint} now reads $\rho^2 = \sum_{i=1}^n(a_i + b_i)(\gamma^\star - \lambda_i)^{-2}$. Since $a_i, b_i\geq 0$, $\lambda_1$ is the largest eigenvalue of $\Psi$ and $\gamma^\star > \lambda_1$, we have that $\rho^2 \leq \sum_{i=1}^n(a_i + b_i)(\gamma^\star - \lambda_1)^{-2}$, yielding the upper bound $\gamma^\star \leq \lambda_1 + \rho^{-1}(\sum_{i=1}^n (a_i + b_i))^{1/2}$. On the other hand, since all terms in the sum $\sum_{i=1}^n(a_i + b_i)(\gamma^\star - \lambda_i)^{-2}$ are nonnegative, it follows that $\rho^2 \geq (a_1 + b_1)(\gamma^\star - \lambda_1)^{-2}$, yielding the lower bound $\gamma^\star \geq \lambda_1 + \rho^{-1}(a_1 + b_1)^{1/2}$. \hfill \QED

\vspace{0.5em}

\textit{Proof of Proposition \ref{prop:mean_covariance_given_optimal_controller}:}
Under Assumption \ref{assm:process_noise_cost}, the unique characterization of the best response $(\mu^\star, M^\star)$ to $(q,U)$ follows from Proposition \ref{prop:linear_moment_objective_gelbrich_constraint} with $\Psi = K^\top(U)K(U)$ and $\psi = K^\top(U)Lq$ using the same argument as in the proof of Proposition \ref{prop:ub}. The proof of the second assertion is analogous to the proof of the second assertion in Proposition \ref{prop:optimal_controller_given_mean_covariance}. \hfill \QED

\vspace{0.5em}

\textit{Proof of Proposition \ref{prop:gradient_computation}:}
% The proof uses a generalization of Danskin's theorem for coercive functions \cite[Theorem 4.13]{bonnans2013perturbation}.
It holds that $J(q, U, \cdot, \cdot)$ is differentiable on $\mathcal{M}^+$, while $J$ and $\nabla_{(\mu,M)}J$ are continuous on $\mathcal{U} \times \mathcal{M}^+$. Moreover, $J(\cdot, \cdot, \mu, M)$ is uniformly coercive on $\mathcal{M}^+$. Indeed, for all $\delta \in (0,1)$, up to a multiplicative positive constant, $J(q, U, \mu, M)$ is lower bounded by $\lambda_{\min}(\hat{\Sigma})\|UF\|_\mathrm{F}^2 + \|UF \mu + q\|_\mathrm{2}^2 \geq \lambda_{\min}(\hat{\Sigma})\|UF\|_\mathrm{F}^2 - (1/\delta - 1)\|UF\|_\mathrm{F}^2 \|\mu\|_2^2 + (1-\delta)\|q\|_2^2$, where the inequality can be shown from Young's inequality; details are omitted. Since $\mu$ is bounded and $\lambda_{\min}(\hat\Sigma) > 0$, uniform coercivity follows by selecting $\delta < 1$ large enough. Hence, $J$ satisfies the inf-compactness assumption in \cite[Theorem~4.13]{bonnans2013perturbation}, and it follows that $V$ has the directional derivative $V'_{(d,D)}(\mu,M) = \inf_{(q, U) \in \mathcal{U}^\star} \langle \nabla_{(\mu,M)}J(q, U, \mu, M), (d, D) \rangle$
% for $(d, D) \in \mathbb{R}^{N_\xi}\times\mathbb{R}^{N_\xi \times N_\xi}$,
where $\langle \cdot, \cdot\rangle$ denotes the inner product on $\mathbb{R}^{N_\xi}\times\mathbb{R}^{N_\xi \times N_\xi}$ and $\mathcal{U}^\star := \argmin_{(q,U) \in \mathcal{U}}J(q, U, \mu, M)$. By Proposition \ref{prop:optimal_controller_given_mean_covariance}, $\mathcal{U}^*$ contains the single element $(q^\star, U^\star)$. Hence, $V'_{(d,D)}(\mu,M) = \langle \nabla_{(\mu,M)}J(q^\star, U^\star, \mu, M), (d, D) \rangle$, which implies that $\nabla_{(\mu, M)} V(\mu, M) = \nabla_{(\mu, M)} J(q^\star, U^\star, \mu, M)$, and the statement follows by computation.\hfill \QED

\end{document}